\documentclass[a4paper,english,sumlimits,reqno]{amsart}

\usepackage[T1]{fontenc}
\usepackage[utf8]{inputenc}
\usepackage{lmodern}

\usepackage{babel}

\usepackage{xargs}

\usepackage{amssymb}
\usepackage{amsmath}
\usepackage{amsthm}
\usepackage{mathtools}
\usepackage{bbm}
\usepackage{bm}
\usepackage{amsbsy}

\usepackage[mathcal]{eucal}
\usepackage{mathrsfs}

\usepackage{stmaryrd}

\usepackage[all]{xy}
\newcommand{\vcxymatrix}[1]{\vcenter{\xymatrix{#1}}}

\usepackage[section]{placeins}
\usepackage{float}

\usepackage{enumerate}

\newcounter{proof}
{\stepcounter{proof}\begin{proof}}%
{\end{proof}}%
\newcounter{proofstep}[proof]
{\refstepcounter{proofstep}\bigskip\par\noindent%
  \ifthenelse{\isempty{#1}}
    {\textbf{Step \theproofstep.}}
    {\textbf{#1.}}
  \noindent}%
{\par}%
\newcounter{proofcase}[proof]
{\refstepcounter{proofcase}\bigskip\par\noindent%
  \ifthenelse{\isempty{#1}}
    {\textbf{Case \theproofcase.}}
    {\textbf{#1.}}
  \noindent}%
{\par}%

\usepackage{hyperref}

\theoremstyle{plain}
\newtheorem{thm}{Theorem}[section]
\newtheorem*{thm*}{Theorem}
\newtheorem{pro}[thm]{Proposition}

\newtheorem{lem}[thm]{Lemma}

\theoremstyle{definition}
\newtheorem{dfn}[thm]{Definition}
\theoremstyle{remark}

\newtheorem{rem}[thm]{Remark}

\numberwithin{equation}{section}

\newcommandx{\textref}[2][1=]{(\hyperref[#2]{#1\ref*{#2}})}

\newcommand{\bfpar}[1]{\smallskip\paragraph{\textbf{#1}}}

\DeclareMathOperator{\union}{\cup}
\DeclareMathOperator*{\Union}{\bigcup}
\DeclareMathOperator{\isect}{\cap}

\newcommand{\cc}[1]{\ensuremath{\llbracket #1 \rrbracket}}

\newcommand{\dint}{\mathscr D}
\newcommand{\drec}{\mathscr R}

\newcommand{\dif}{\ensuremath{\, \mathrm d}}

\DeclareMathOperator{\sign}{sign}
\DeclareMathOperator{\Id}{Id}

\DeclareMathOperator{\spn}{span}

\DeclareMathOperator{\cond}{\mathbb E}

\DeclareMathOperator{\lesslex}{<_\ell}

\DeclareMathOperator{\drless}{\vartriangleleft\,}

\newcommand{\drindex}{\ensuremath{\mathcal O_{\drless}}}

\newcommand{\bmo}{\ensuremath{\mathrm{BMO}}}

\begin{document}

\title[Factorization in mixed norm Hardy and BMO spaces]
{Factorization in mixed norm Hardy and BMO spaces}

\author[R.~Lechner]{Richard Lechner}
\address{Richard Lechner,
  Institute of Analysis,
  Johannes Kepler University Linz,
  Altenberger Strasse 69,
  A-4040 Linz, Austria}
\email{Richard.Lechner@jku.at}

\date{\today}
\thanks{Supported by the Austrian Science Foundation (FWF) Pr.Nr. P28352}

\begin{abstract}
  Let $1\leq p,q < \infty$ and $1\leq r \leq \infty$.
  We show that the direct sum of mixed norm Hardy spaces $\big(\sum_n H^p_n(H^q_n)\big)_r$ and the
  sum of their dual spaces $\big(\sum_n H^p_n(H^q_n)^*\big)_r$ are both primary.
  We do so by using Bourgain's localization method and solving the finite dimensional factorization
  problem.
  In particular, we obtain that the spaces
  $\big(\sum_{n\in \mathbb N} H_n^1(H_n^s)\big)_r$,
  $\big(\sum_{n\in \mathbb N} H_n^s(H_n^1)\big)_r$,
  as well as
  $\big(\sum_{n\in \mathbb N} \bmo_n(H_n^s)\big)_r$
  and $\big(\sum_{n\in \mathbb N} H^s_n(\bmo_n)\big)_r$,
  $1 < s < \infty$, $1\leq r \leq \infty$,
  are all primary.
\end{abstract}

\maketitle

\section{Introduction}\label{sec:intro}

\noindent
Let $\dint$ denote the collection of dyadic intervals on the unit interval, which is given by
\begin{equation*}
  \dint
  = \{[k2^{-n},(k+1)2^{-n}) : n,k\in \mathbb N_0, 0\leq k \leq 2^n-1\}.
\end{equation*}
The dyadic intervals are nested, i.e. if $I,J\in \dint$, then $I\cap J \in \{I,J,\emptyset\}$.
For $I\in \dint$ we let $|I|$ denote the length of the dyadic interval $I$.
The Carleson constant $\cc{\mathscr C}$ of a collection $\mathscr C\subset \dint$ is given by
\begin{equation*}
  \cc{\mathscr C}
  = \sup_{I\in \mathscr C} \frac{1}{|I|} \sum_{\substack{J\in \mathscr C\\J\subset I}} |J|.
\end{equation*}
Let $I\in \dint$ and $I\neq [0,1)$, then $\widetilde I$ is the unique dyadic interval
satisfying $\widetilde I\supset I$ and $|\widetilde I| = 2 |I|$.
Given $N_0\in \mathbb N_0$ we define
\begin{equation*}
  \dint_{N_0} = \{I\in \dint : |I| = 2^{-N_0}\}
  \qquad\text{and}\qquad
  \dint^{N_0} = \{I\in \dint : |I| \geq 2^{-N_0}\}.
\end{equation*}
Let~$h_I$ be the $L^\infty$--normalized Haar function supported on $I\in\dint$;
that is, $h_I$ is $+1$ on the left half of $I$, it is $-1$ on the right half of $I$, and zero
otherwise.
For $1\leq p < \infty$, the \emph{Hardy space} $H^p$ is the completion of
\begin{equation*}
  \spn\{ h_I : I \in \dint \}
\end{equation*}
under the square function norm
\begin{equation}\label{eq:Hp-norm}
  \|f\|_{H^p}
  = \Big(
      \int_0^1 \big(
        \sum_I |a_I|^2 h_I^2(x)
      \big)^{p/2}
    \dif x
    \Big)^{1/p}
  ,
\end{equation}
where $f = \sum_I a_I h_I$.

Let $\drec = \{I\times J: I,J\in\dint\}$ denote the collection of dyadic rectangles contained in the
unit square, and define the bi--parameter $L^\infty$--normalized Haar system by
\begin{equation*}
  h_{I\times J}(x,y) = h_I(x)h_J(y),
  \qquad I\times J\in\drec,\, x,y\in[0,1).
\end{equation*}
For $1\leq p,q < \infty$, the \emph{mixed-norm Hardy space} $H^p(H^q)$ is the completion of
\begin{equation*}
  \spn\{ h_{I\times J} : I\times J \in \drec \}
\end{equation*}
under the square function norm
\begin{equation}\label{eq:HpHq-norm}
  \|f\|_{H^p(H^q)}
  = \bigg(
    \int_0^1 \Big(
      \int_0^1 \big(
        \sum_{I\times J} |a_{I\times J}|^2 h_{I\times J}^2(x,y)
      \big)^{q/2}
    \dif y
    \Big)^{p/q}
    \dif x
  \bigg)^{1/p}
  ,
\end{equation}
where $f = \sum_{I\times J} a_{I\times J} h_{I\times J}$.
Given $m,n\in \mathbb N$, we define the space $H^p_m(H^q_n)$ by
\begin{equation*}
  H^p_m(H^q_n) = \spn\{ h_{I\times J} : I\in \dint^m, J\in \dint^n\},
\end{equation*}
equipped with the norm $\|\cdot\|_{H^p(H^q)}$.

For the following elementary and well known facts for which we refer
to~\cite{lindenstrauss-tzafriri:1977,maurey:1980,capon:1982,wojtaszczyk:1991,mueller:1994,mueller:2005} as sources:
\begin{itemize}
\item $(h_{I\times J})_{I\times J\in\drec}$ is an unconditional basis of~$H^p(H^q)$, called the
\emph{bi--parameter Haar system}.
This basis is $L^\infty$--normalized and not normalized in $H^p(H^q)$; in fact, we have
$\|h_{I\times J}\|_{H^p(H^q)} = |I|^{1/p}|J|^{1/q}$.

\item Let $1\leq p,q < \infty$ and let $H^p(H^q)^*$ denote the dual space of $H^p(H^q)$ with the
  usual operator norm given by
  \begin{equation}\label{eq:dual-norm}
    \|g\|_{H^p(H^q)^*} = \sup \{ |\langle g, f\rangle| : \|f\|_{H^p(H^q)}\leq 1 \}.
  \end{equation}

\item Since $h_{I\times J}$, $I\times J$ is a Schauder basis in $H^p(H^q)$, we canonically identify
  the elements $g\in H^p(H^q)^*$ with the sequence $(\langle g, h_{I\times J}\rangle)_{I\times J}$.
  Moreover, as $h_{I\times J}$, $I\times J$ is a $1$--unconditional basis in $H^p(H^q)$, the norm of
  $(|\langle g, h_{I\times J}\rangle|)_{I\times J}$ is equal to the norm of
  $(\langle g, h_{I\times J}\rangle)_{I\times J}$,
  see~\cite[Chapter 1]{lindenstrauss-tzafriri:1977}.

\item We naturally identify $h_{I_0 \times J_0}$ as an element of $H^p(H^q)^*$ by the following
  definition:
  $\langle h_{I_0\times J_0}, h_{I\times J} \rangle = |I\times J|$ if $I\times J = I_0\times J_0$,
  and $\langle h_{I_0\times J_0}, h_{I\times J} \rangle = 0$ if $I\times J \neq I_0\times J_0$.

\item Let $1 < p,q < \infty$ and $\frac{1}{p} + \frac{1}{p'} = 1$,
  $\frac{1}{q} + \frac{1}{q'} = 1$.
  Since for any finite linear combination of Haar functions $f$ we have
  \begin{equation*}
    C_{p,q}^{-1} \|f\|_{L^p(L^q)} \leq \|f\|_{H^p(H^q)} \leq C_{p,q} \|f\|_{L^p(L^q)},
  \end{equation*}
  the identity operator provides an isomorphism between $H^p(H^q)$ and $L^p(L^q)$.
  Hence, the dual of $H^p(H^q)$ identifies with $H^{p'}(H^{q'})$.
  Similarily, for the limiting cases we have $H^1(H^q)^* = \bmo(H^{q'})$,
  $H^p(H^1)^* = H^{q'}(\bmo)$ and $H^1(H^1)^* = \bmo(\bmo)$.
  See~\cite{maurey:1980} and also \cite{mueller:1994}.

\item The isomorphisms that identify the duals of the mixed norm Hardy spaces
  $H^p(H^q)$, $1\leq p,q < \infty$ with the spaces mentioned above may or may not depend on $p$
  and~$q$ (the uncertainty stems from not specifying the norm of $\bmo$).
  Since the constants in our results do not depend on $p$ or~$q$, but some of the proofs involve the
  dual space of $H^p(H^q)$, we have to be careful not to introduce dependencies on $p$ and~$q$ this
  way.
  As we will see, all the proofs are carried out by strictly using the dual norm of $H^p(H^q)$
  specified in~\eqref{eq:dual-norm}, thereby avoiding the problem of introducing $p$ or~$q$
  dependencies in the estimates.
\end{itemize}

Let $k,m,n\in \mathbb N$ and $1\leq p,q < \infty$.
Let $(b_i : 1\leq i \leq k)$ denote a block basis of bi--parameter Haar functions in
$H_m^p(H_n^q)$, and let $(b_i^* : 1\leq i \leq k)$ denote the bi--orthogonal functions, i.e.
$\langle b_i^*, b_i \rangle = 1$, and $\langle b_i^*, b_j \rangle = 0$, if $i\neq j$.
We say an operator $T : H_m^p(H_n^q)\to H_m^p(H_n^q)$ \emph{has large diagonal} with respect to the
system $(b_i : 1\leq i \leq k)$, if there exists a $\delta > 0$ such that
$|\langle b_i^*, T b_i\rangle| > \delta$, for all $1 \leq i \leq k$, and $\delta$ does not depend
on any of $m,n,k,p,q$.

We will briefly state the version of Pe{\l}czy{\'n}ski's decomposition method that we will use here:
Let $X$ and $Y$ be Banach spaces so that $X$ is isomorphic to a complemented subspace of $Y$, and
vice versa.
If $X$ is such that $X$ is isomorphic to $(\sum X)_r$ for some $1\leq r \leq \infty$, then $X$ is
isomorphic to $Y$.

The main object that we will study are the spaces
$\big(\sum_{m,n\in \mathbb N} H^p_m(H^q_n)\big)_r$, $1\leq p,q < \infty$ and $1\leq r \leq \infty$.
They are defined as follows:
\begin{equation}\label{eq:double-sum}
  \big(\sum_{m,n\in \mathbb N} H^p_m(H^q_n)\big)_r
  = \{ f = (f_{m,n})_{m,n\in \mathbb N} : f_{m,n}\in H_m^p(H_n^q),\ \|f\|_r < \infty \},
\end{equation}
equipped with the norm $\|f\|_r$ given by
\begin{equation}\label{eq:double-sum-norm}
  \|f\|_r
  = \big( \sum_{m,n\in \mathbb N} \|f_{m,n}\|_{H^p(H^q)}^r \big)^{1/r}, \text{if $r < \infty$},
  \ \text{and}\ 
  \|f\|_\infty = \sup_{m,n}\|f_{m,n}\|_{H^p(H^q)}.
\end{equation}
Naturally, the question arises how many non--isomorphic spaces are defined by~\eqref{eq:double-sum}
and~\eqref{eq:double-sum-norm}.
In Proposition~\ref{pro:iso-iso} we assert that
\begin{equation}\label{eq:results:1}
  \big(\sum_{m,n\in \mathbb N} H^p_m(H^q_n)\big)_r
  \quad\text{is isometrically isomorphic to}\quad
  \big(\sum_{n\in \mathbb N} H^p_n(H^q_n)\big)_r,
\end{equation}
and that by a variant of Pitt's theorem (see Theorem~\ref{thm:pitt}) the spaces
\begin{equation}\label{eq:results:2}
  \big(\sum_{n\in \mathbb N} H^p_n(H^q_n)\big)_r
  \ \text{and}\ 
  \big(\sum_{n\in \mathbb N} H^p_n(H^q_n)\big)_s
  \ \text{are not isomorphic for $1\leq r\neq s\leq \infty$}.
\end{equation}

\section{Main results}\label{sec:results}

\noindent
Here, we state the main results Theorem~\ref{thm:factorization} and Theorem~\ref{thm:primary} and describe the
concept of proof.
Their respective proofs are carried out in Section~\ref{sec:diag-glue-sums}.

\begin{thm}\label{thm:factorization}
  Let $1\leq p,q < \infty$ and $1\leq r \leq \infty$,
  and for all $n\in \mathbb N$ let $X_n$ denote the space $H^p_n(H^q_n)$ or its dual
  $H^p_n(H^q_n)^*$.
  For any $\eta > 0$ and any operator
  $T : \big(\sum_{n\in \mathbb N} X_n\big)_r\to
  \big(\sum_{n\in \mathbb N} X_n\big)_r$,
  there exist operators
  $R,S : \big(\sum_{n\in \mathbb N} X_n\big)_r\to \big(\sum_{n\in \mathbb N} X_n\big)_r$
  such that
  \begin{equation}\label{eq:thm:factorization}
    \vcxymatrix{
      \big(\sum_{n\in \mathbb N} X_n\big)_r \ar[r]^{\Id} \ar[d]_S
      & \big(\sum_{n\in \mathbb N} X_n\big)_r\\
      \big(\sum_{n\in \mathbb N} X_n\big)_r \ar[r]_H
      & \big(\sum_{n\in \mathbb N} X_n\big)_r \ar[u]_R
    }
  \end{equation}
  for $H = T$ or $H = \Id - T$ and $\|R\| \|S\| \leq 2 + \eta$.
\end{thm}

There are several factorization results of the form~\eqref{eq:thm:factorization} regarding
bi--parameter Hardy spaces or their duals; among them are the following:
\begin{itemize}
\item $H^p(H^q)$ for $1 < p,q < \infty$, see~\cite{capon:1982}.

\item $H^1(H^1)$, see~\cite{mueller:1994}.

\item $\big(\sum_{n\in \mathbb N} H_n^1(H_n^1)\big)_r$
  and $\big(\sum_{n\in \mathbb N} \bmo_n(\bmo_n)\big)_r$, $1\leq r \leq \infty$,
  see~\cite{lechner_mueller:2014}.
  We remark that the latter result is only stated explicitly for the space
  $\big(\sum_{n\in \mathbb N} \bmo_n(\bmo_n)\big)_\infty$.
  The other assertions follow immediately from reasonable modifications of the proof given
  in~\cite{lechner_mueller:2014}.
\end{itemize}

Theorem~\ref{thm:factorization} is an extension of the above list.
Specifically, Theorem~\ref{thm:factorization} yields factorization results (with possibly different
constants as discussed in the introduction) for the following spaces:
\begin{equation*}
  \big(\sum_{n\in \mathbb N} H_n^1(H_n^s)\big)_r,
  \qquad \big(\sum_{n\in \mathbb N} H_n^s(H_n^1)\big)_r,
  \quad \big(\sum_{n\in \mathbb N} \bmo_n(H_n^s)\big)_r,
  \quad \big(\sum_{n\in \mathbb N} H^s_n(\bmo_n)\big)_r,
\end{equation*}
where $1 < s < \infty$, $1\leq r \leq \infty$.
The proof of Theorem~\ref{thm:factorization} is based on Bourgain's localization
method~\cite{bourgain:1983} and consists of the following three major steps:
\begin{itemize}
\item Reduction to diagonal operators.

\item Proving the following quantitative factorization problem:
  For all $n \in \mathbb N$ and $\Gamma, \eta > 0$ there exists an integer $N=N(n,\Gamma,\eta)$ so
  that the following holds:
  For any operator $D : X_N\rightarrow X_N$ with $\|D\|\leq \Gamma$ there exist
  operators $R,S$ so that for either $H=D$ or $H = \Id_{X_n} - D$ we have that
  \begin{equation}\label{eq:local-factorization}
    \vcxymatrix{X_n \ar[r]^{\Id} \ar[d]_S & X_n\\
      X_N \ar[r]_H & X_N \ar[u]_R}
    \qquad\text{where $\|R\|\|S\|\leq 2+\eta$}.
  \end{equation}

\item Glueing the finite dimensional pieces~\eqref{eq:local-factorization} together to obtain the
  factorization diagram~\eqref{eq:thm:factorization}.
\end{itemize}

Before we come to the next result, let us recall the notion of a primary Banach space, see e.g.~\cite{lindenstrauss-tzafriri:1977}:
A Banach space $X$ is \emph{primary} if for every bounded projection $Q : X\to X$, either $Q(X)$ or
$(\Id - Q)(X)$ is isomorphic to $X$.
An immediate consequence of Theorem~\ref{thm:factorization} is the subsequent Theorem~\ref{thm:primary}.
\begin{thm}\label{thm:primary}
  Let $1\leq p,q < \infty$ and $1\leq r \leq \infty$.
  Then $\big(\sum_{n\in \mathbb N} H^p_n(H^q_n)\big)_r$ and
  $\big(\sum_{n\in \mathbb N} H^p_n(H^q_n)^*\big)_r$ are primary.
  In particular, the following spaces are primary:
  \begin{equation*}
    \big(\sum_{n\in \mathbb N} H_n^1(H_n^s)\big)_r,
    \quad \big(\sum_{n\in \mathbb N} H_n^s(H_n^1)\big)_r,
    \quad \big(\sum_{n\in \mathbb N} \bmo_n(H_n^s)\big)_r,
    \quad \big(\sum_{n\in \mathbb N} H^s_n(\bmo_n)\big)_r,
  \end{equation*}
  where $1 < s < \infty$.
\end{thm}
Note that if $p=q=1$, Theorem~\ref{thm:primary} follows from~\cite{lechner_mueller:2014}.

\section{The local product conditions}\label{sec:clpc}\hfill

\noindent
In Section~\ref{subsec:statement-capons-local} and Section~\ref{subsec:consequences-capons-local} we discuss
conditions (the local product conditions~\textref[P]{enu:p1}--\textref[P]{enu:p4}) under which a
block basis of the bi--parameter Haar system is equivalent to the bi--parameter Haar system, and
that the orthogonal projection onto this block basis is bounded in $H^p(H^q)$.
Section~\ref{subsec:statement-capons-local} and~\ref{subsec:consequences-capons-local} are a
compilation of definitions and results of~\cite{laustsen:lechner:mueller:2015}.
Section~\ref{sec:reit-capons-local} is new; it contains the result that reiterating the local product conditions yields again the local product conditions.
The local product conditions were modeled after Capon's conditions isolated in~\cite{capon:1982}.

\subsection{Statement of the local product conditions}\label{subsec:statement-capons-local}\hfill

\noindent
Let $\mathscr A\subset \drec$ be an index set.
For each $R\in \mathscr A$ let $\mathscr X_R, \mathscr Y_R\subset \dint$ denote non-empty
collections of dyadic intervals that define the collection of dyadic rectangles $\mathscr B_R$ by
\begin{equation}\label{eq:symbols-0}
  \mathscr B_R
  = \mathscr X_R\times \mathscr Y_R
  = \{K\times L\, :\, K\in \mathscr X_R,\, L\in \mathscr Y_R\},
  \qquad R\in \mathscr A,
\end{equation}
see Figure~\ref{fig:2dhaar-layout} and\ref{fig:2dhaar-one-step}.
\begin{figure}[bht]
  \begin{center}
    \includegraphics[scale=0.25]{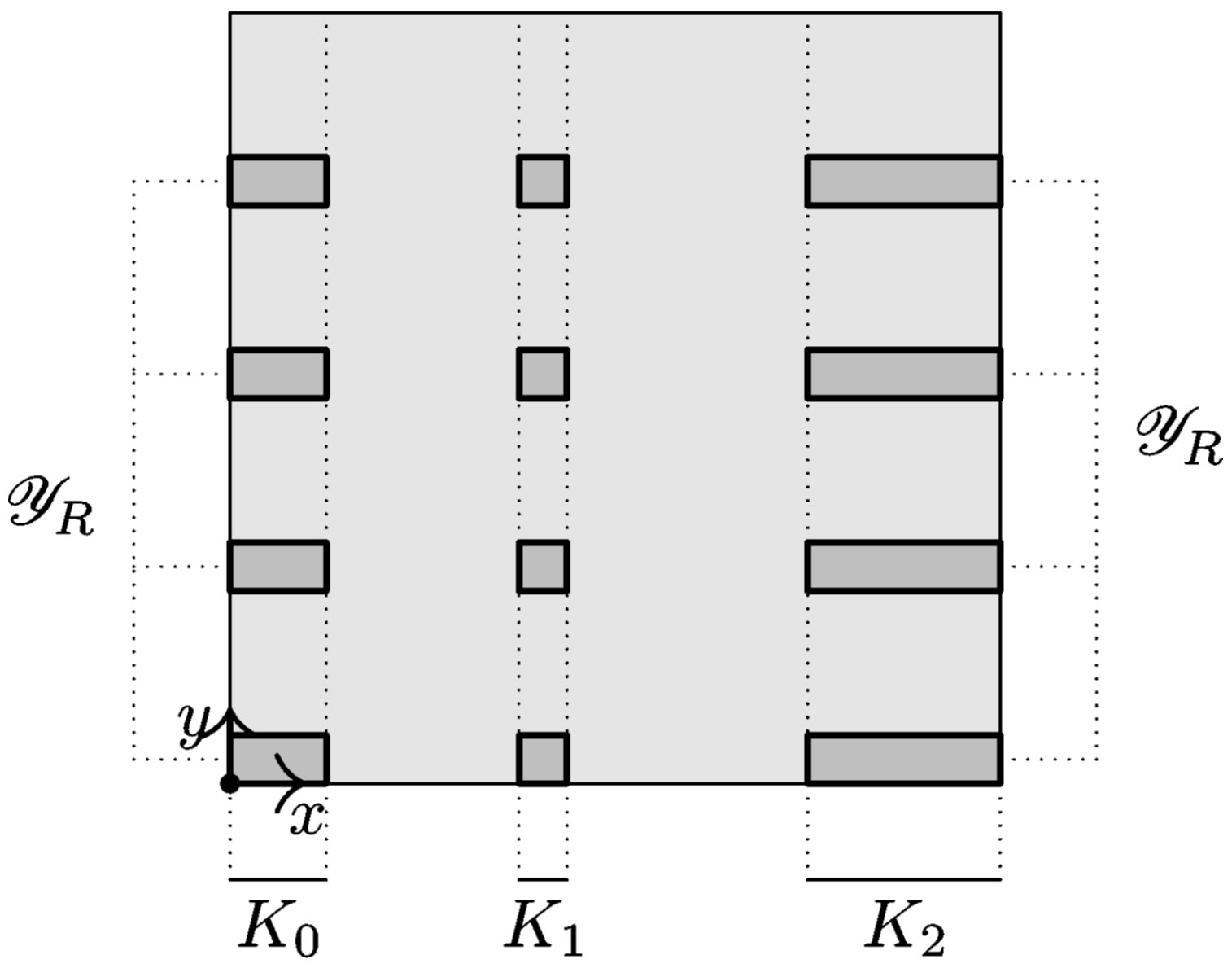}
  \end{center}
  \caption{Given a dyadic index rectangle $R\in \mathscr A$, this figure depicts
    the collection of dark gray rectangles $\mathscr B_R$, which are contained in the light gray
    unit square.
    $\mathscr B_R$ is of the form $\mathscr B_R = \mathscr X_R\times \mathscr Y_R$, where
    $\mathscr X_R = \{K_0,K_1,K_2\}$.
  }
  \label{fig:2dhaar-layout}
\end{figure}
\begin{figure}[bht]
  \begin{center}
    \includegraphics[scale=0.225]{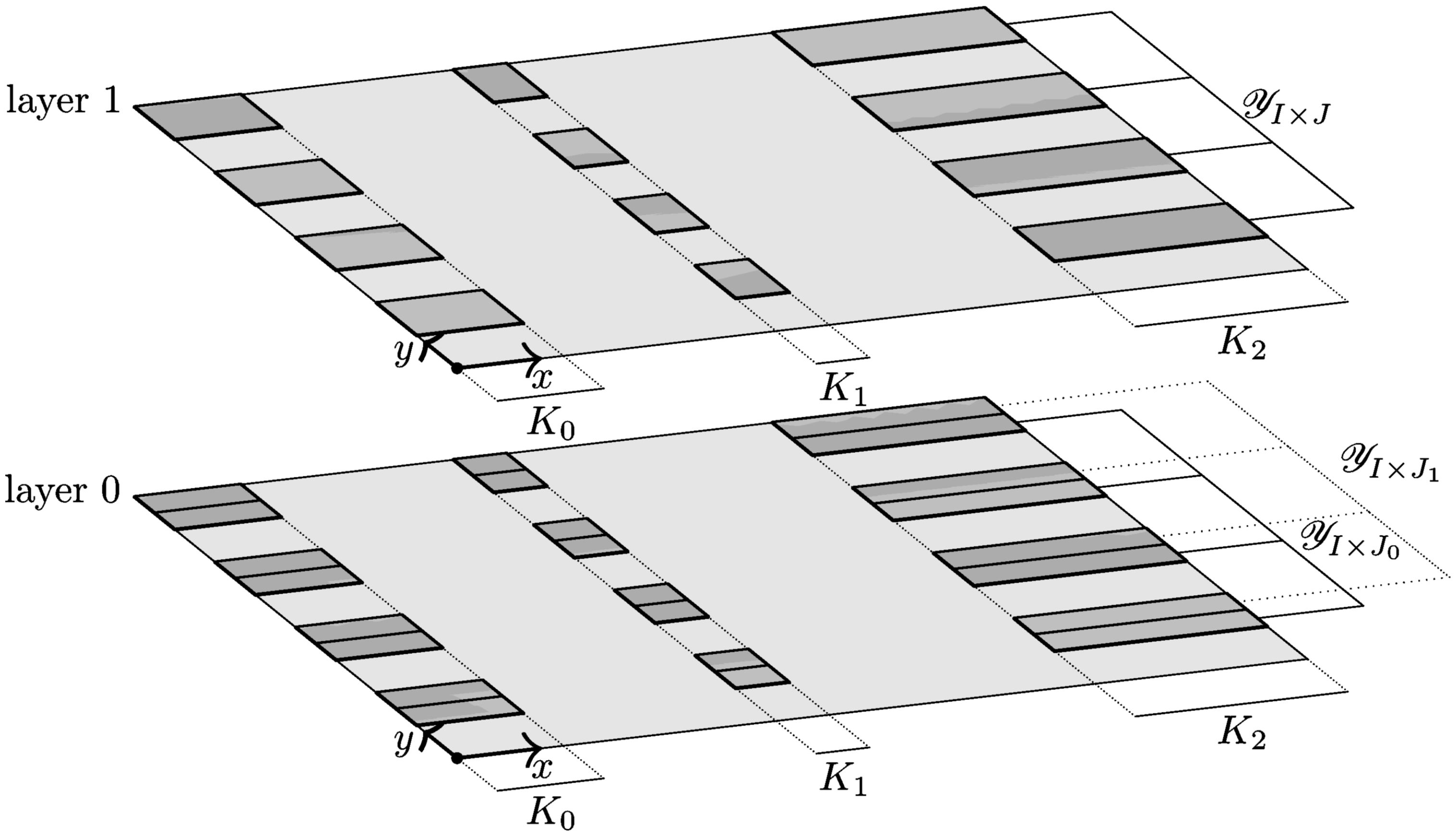}
  \end{center}
  \caption{The dyadic index rectangles
    $I\times J$, $I\times J_0$ and $I\times J_1$ in $\mathscr A$ are such that
    $J_0\union J_1 = J$ and $J_0\cap J_1 = \emptyset$.
    This figure depicts the collections
    $\mathscr B_{I\times J} = \mathscr X_{I\times J}\times \mathscr Y_{I\times J}$ in the top layer
    (see also Figure~\ref{fig:2dhaar-layout}), and
    $\mathscr B_{I\times J_0} = \mathscr X_{I\times J_0}\times \mathscr Y_{I\times J_0}$ and
    $\mathscr B_{I\times J_1} = \mathscr X_{I\times J_1}\times \mathscr Y_{I\times J_1}$ in the
    bottom layer, where $\mathscr X_{I\times J} = \{K_0,K_1,K_2\}$.
    Each interval in $\mathscr Y_{I\times J}$ is split in two intervals, which are then placed into
    $\mathscr Y_{I\times J_0}$ and $\mathscr Y_{I\times J_1}$, respectively.
  } 
  \label{fig:2dhaar-one-step}
\end{figure}
For all $R\in \mathscr A$ and $x,y\in [0,1)$ we define
\begin{equation}\label{eq:block-basis}
  b_R(x,y)
  = \sum_{K\times L\in \mathscr B_R} h_{K\times L}(x,y)
  = \Big( \sum_{K\in \mathscr X_R} h_K(x) \Big)
  \Big( \sum_{L\in \mathscr Y_R} h_L(y) \Big).
\end{equation}
For the second equality in~\eqref{eq:block-basis} see~\eqref{eq:symbols-0}.
We call $(b_R : R\in \drec)$ the \emph{block basis generated by} $(\mathscr B_R : R\in \drec)$.

We now introduce some notation.
For $R \in\mathscr A$ we set
\begin{equation}\label{eq:symbols-1}
  X_R = \bigcup\{ K : K\in \mathscr X_R \}
  \qquad\text{and}\qquad
  Y_R = \bigcup\{L : L\in \mathscr Y_R\}.
\end{equation}
For each $I_0\times J_0\in \mathscr A$ we take the following unions:
\begin{equation}\label{eq:symbols-2}
  X_{I_0}
  = \bigcup\{ X_{I_0\times J} : I_0\times J\in \mathscr A\},
  \quad
  Y_{J_0}
  = \bigcup\{ Y_{I\times J_0} : I\times J_0\in \mathscr A\}.
\end{equation}
By~\eqref{eq:symbols-2} we have that for all $I\times J\in \mathscr A$
\begin{equation}\label{eq:XY-inclusions}
  X_{I\times J} \subset X_I
  \qquad\text{and}\qquad
  Y_{I\times J} \subset Y_J.
\end{equation}

We say that $\{\mathscr B_{I\times J}: I\times J\in\mathscr A\}$ satisfies the
\emph{local product conditions} with constants $C_X,C_Y>0$, if the following four
conditions~\textref[P]{enu:p1}, \textref[P]{enu:p2}, \textref[P]{enu:p3} and~\textref[P]{enu:p4}
hold.
\begin{enumerate}[\quad(P1)]
\item\label{enu:p1}%
  For all $R\in \mathscr A$ the collection $\mathscr B_R$ consists of pairwise
  disjoint dyadic rectangles, and for all $R_0,R_1\in\mathscr A$ with $R_0\neq R_1$ we have
  $\mathscr B_{R_0} \isect \mathscr B_{R_1} = \emptyset$.
\item\label{enu:p2}%
  For all $I\times J, I_0\times J_0, I_1\times J_1\in\mathscr A$ with
  $I_0 \cap I_1 = \emptyset$, $I_0\union I_1\subset I$ and
  $J_0 \cap J_1 = \emptyset$, $J_0\union J_1\subset J$ we have the inclusions
  \begin{align*}
    X_{I_0}\isect X_{I_1} &= \emptyset &X_{I_0}\union X_{I_1} & \subset X_I,\\
    Y_{J_0}\isect Y_{J_1} &= \emptyset &Y_{J_0}\union Y_{J_1} & \subset Y_J.
  \end{align*}
\item\label{enu:p3}%
  For all $I\times J\in\mathscr A$, we have
  \begin{equation*}
    C_X^{-1} |I| \leq |X_I| \leq C_X |I|
    \qquad\text{and}\qquad
    C_Y^{-1} |J| \leq |Y_J| \leq C_Y |J|.
  \end{equation*}
\item \label{enu:p4}%
  For all $I_0\times J_0, I\times J\in \mathscr A_1$ with $I_0\times J_0\subset I\times J$ and
  for every $K\in \mathscr X_{I\times J}$ and $L\in \mathscr Y_{I\times J}$, we have
  \begin{equation*}
    \frac{|K\isect X_{I_0}|}{|K|} \geq  C_X^{-1}\frac{|X_{I_0}|}{|X_I|}
    \qquad\text{and}\qquad
    \frac{|L\isect Y_{J_0}|}{|L|} \geq  C_Y^{-1}\frac{|Y_{J_0}|}{|Y_J|}.
  \end{equation*}
\end{enumerate}

\subsection{Implications from the local product
  conditions}\label{subsec:consequences-capons-local}\hfill

\noindent
The local product conditions~\textref[P]{enu:p1}--\textref[P]{enu:p4} ensure that the block basis
$(b_R : R \in \mathscr A)$ given by~\eqref{eq:block-basis} is equivalent to the Haar system
$(h_R : R \in \mathscr A)$ in $H^p(H^q)$, for $1 < p, q < \infty$, and that the orthogonal
projection $Q^{(\varepsilon)}$ (see~\eqref{eq:projection-alpha-estimate} below) onto
$(b_R : R \in \mathscr A)$ is bounded in $H^p(H^q)$, $1 < p, q < \infty$.

In particular, Theorem~\ref{thm:projection} includes the following endpoints:
\begin{equation*}
  H^1(H^s),
  \qquad
  H^s(H^1),
  \qquad
  H^s(\bmo),
  \qquad
  \bmo(H^s),
  \qquad 1 < s < \infty.
\end{equation*}
Theorem~\ref{thm:projection} is taken from~\cite{laustsen:lechner:mueller:2015}.
\begin{thm}\label{thm:projection}
  Let $1 \leq p,q < \infty$.
  Assume that $(\mathscr B_R\, :\, R\in \drec)$ satisfies the local product
  conditions~\textref[P]{enu:p1}--\textref[P]{enu:p4} with constants~$C_X$ and~$C_Y$.
  Let $\varepsilon = (\varepsilon_Q : Q\in \drec)$ be a scalar sequence with $|\varepsilon_Q|=1$,
  and let the block basis of the bi--parameter Haar system $b_R^{(\varepsilon)}$ be given by
  \begin{equation}\label{eq:block-basis-alpha}
    b_R^{(\varepsilon)} = \sum_{Q\in \mathscr B_R} \varepsilon_Q h_Q,
    \qquad R\in \drec.
  \end{equation}
  Then the following assertions are true:
  \begin{enumerate}[(i)]
  \item For all sequences of scalars $a_R$, $R\in \drec$, we have that
    \begin{equation}\label{eq:local-product-condition-i}
      C^{-1} \Big\|\sum_R a_R h_R\Big\|
      \leq \Big\|\sum_R a_R b_R^{(\varepsilon)} \Big\|
      \leq C \Big\|\sum_R a_R h_R\Big\|.
    \end{equation}
    The above norms are either all the norm $\|\cdot\|_{H^p(H^q)}$, or they are all the norm
    $\|\cdot\|_{H^p(H^q)^*}$.
    \label{enu:local-product-condition-i}

  \item The orthogonal projection $Q^{(\varepsilon)}$ given by
    \begin{equation}\label{eq:projection-alpha}
      Q^{(\varepsilon)} f = \sum_{R\in \drec}
      \frac{\langle b_R^{(\varepsilon)}, f\rangle}{\|b_R^{(\varepsilon)}\|_2^2}
      b_R^{(\varepsilon)}
    \end{equation}
    satisfies the estimates
    \begin{equation}\label{eq:projection-alpha-estimate}
      \begin{aligned}
        \|Qf\|_{H^p(H^q)} & \leq C \|f\|_{H^p(H^q)},
        & f &\in H^p(H^q),\\
        \|Qf\|_{H^p(H^q)^*} & \leq C \|f\|_{H^p(H^q)^*},
        & f &\in H^p(H^q)^*.
      \end{aligned}
    \end{equation}
    \label{enu:local-product-condition-iii}
  \end{enumerate}
  There exists a universal integer $k$ such that $C \leq C_X^k C_Y^k$.
\end{thm}

We remark that by~\eqref{eq:symbols-0}, we can rewrite~\eqref{eq:block-basis-alpha} in the following ways:
\begin{equation}\label{eq:block-basis-alpha-local-product}
  \begin{aligned}
    b_R^{(\varepsilon)}(x,y) 
    &= \sum_{K\in \mathscr X_R} h_K(x) \sum_{L\in \mathscr Y_R} \varepsilon_{K\times L} h_L(y),
    \qquad R\in \drec,\\
    b_R^{(\varepsilon)}(x,y) 
    &= \sum_{L\in \mathscr Y_R} h_L(y) \sum_{K\in \mathscr X_R} \varepsilon_{K\times L} h_K(x),
    \qquad R\in \drec.
  \end{aligned}
\end{equation}

\subsection{Reiterating the local product conditions}\label{sec:reit-capons-local}\hfill

\noindent
Let $(e_I)$ and $(f_J)$ denote block bases of the one parameter Haar system, such that
$(e_I\otimes f_J)$ satisfies \textref[P]{enu:p1}--\textref[P]{enu:p4}.
Moreover, we will assume that the regularity assumptions
Lemma~\ref{lem:product:local-product}~\eqref{enu:lem:product:local-product:i}
and~\eqref{enu:lem:product:local-product:ii} are satisfied.
Assumption~\eqref{enu:lem:product:local-product:i} really is a one--parameter version
of \textref[P]{enu:p1}.
The more interesting assumption is~\eqref{enu:lem:product:local-product:ii}, which says that the
inclusion of two index intervals $I_0\subset I$ (respectively $J_0\subset J$) implies the inclusion
of each $E_0\in \mathscr E_{I_0}$ in some $E\in \mathscr E_I$ (respectively of each
$F_0\in \mathscr F_{I_0}$ in some $F\in \mathscr F_I$).
Lemma~\ref{lem:product:local-product} tells us that if one uses $(e_I\otimes f_J)$ instead of the
bi--parameter Haar system $(h_I \otimes h_J)$ to build a bi--parameter block basis according to
\textref[P]{enu:p1}--\textref[P]{enu:p4}, then the result is a block basis of the bi--parameter Haar
system satisfying the local product conditions~\textref[P]{enu:p1}--\textref[P]{enu:p4}.
\begin{lem}\label{lem:product:local-product}
  Let $\mathscr A$ be a collection of index rectangles.
  Let
  \begin{equation*}
    (\mathscr E_I\times \mathscr F_J : I\times J\in \mathscr A)
  \end{equation*}
  be a sequence of dyadic rectangles satisfying \textref[P]{enu:p1}--\textref[P]{enu:p4} with
  constants $C_E$ and $C_F$.
  Moreover, we assume the following:
  \begin{enumerate}[(i)]
  \item For all $I\times J\in \mathscr A$ the collections $\mathscr E_I$ and $\mathscr F_J$ each
    consist of pairwise disjoint intervals, and $\mathscr E_{I_0}\cap \mathscr E_{I_1} = \emptyset$
    for all $I_0\times J_0,I_1\times J_1\in \mathscr A$ with $I_0\neq I_1$, and
    $\mathscr F_{J_0}\cap \mathscr F_{J_1} = \emptyset$, for all
    $I_0\times J_0,I_1\times J_1\in \mathscr A$ with and $J_0\neq J_1$.
    \label{enu:lem:product:local-product:i}
  \item Whenever $I_0\times J_0, I\times J\in \mathscr A$ with $I_0\subset I$ and $J_0\subset J$,
    then
    \begin{align*}
      \text{for all $E_0\in \mathscr E_{I_0}$ there exists an $E\in \mathscr E_I$ such that
      $E_0\subset E$},\\
      \text{for all $F_0\in \mathscr F_{I_0}$ there exists an $F\in \mathscr F_I$ such that
      $F_0\subset F$}.
    \end{align*}
    \label{enu:lem:product:local-product:ii}

  \end{enumerate}
  Let
  \begin{equation*}
    \Big( \mathscr B_{E\times F} : E\times F\in \bigcup_{I\times J\in \mathscr A}
    \mathscr E_I\times \mathscr F_J \Big)
  \end{equation*}
  be a sequence of collections of dyadic rectangles satisfying the local product conditions
  \textref[P]{enu:p1}--\textref[P]{enu:p4} with constants $C_X$ and $C_Y$.
  For each $I\times J\in \mathscr A$ define the collection $\widetilde{\mathscr B}_{I\times J}$ by
  \begin{equation*}
    \widetilde{\mathscr B}_{I\times J}
    = \bigcup_{\substack{E\in \mathscr E_I\\F\in \mathscr F_J}} \mathscr B_{E\times F}.
  \end{equation*}
  Then the sequence of collections $(\widetilde{\mathscr B}_{I\times J})$ satisfies the local
  product conditions \textref[P]{enu:p1}--\textref[P]{enu:p4} with constants $C_EC_X^3$ and
  $C_FC_Y^3$.
\end{lem}

\begin{rem}\label{rem:thm:projection}
  Consequently, Theorem~\ref{thm:projection} applies to the collections $\widetilde{\mathscr B}_R$
  and the block basis of the bi--parameter Haar system $\widetilde b_R^{(\varepsilon)}$ given by
  \begin{equation*}
    \widetilde b_{I\times J}^{(\varepsilon)}
    = \sum_{Q\in \widetilde{\mathscr B}_{I\times J}} \varepsilon_Q h_Q
    = \sum_{E\in \mathscr E_I} \sum_{F\in \mathscr F_J} b_{E\times F}^{(\varepsilon)},
    \qquad I\times J\in \drec,
  \end{equation*}
  where the block basis $(b_{E\times F}^{(\varepsilon)})$ is given by
  \begin{equation*}
    b_{E\times F}^{(\varepsilon)} = \sum_{Q\in \mathscr B_{E\times F}} \varepsilon_Q h_Q,
    \qquad E\times F\in \bigcup_{I\times J\in \mathscr \drec} \mathscr E_I\times \mathscr F_J.
  \end{equation*}
  
\end{rem}

\begin{proof}[Proof of Lemma~\ref{lem:product:local-product}]
  Within this proof, we shall make use of the following convention.
  Whenever there is an indentifier of an object which uses the script font (i.e. $\mathscr Z$), then
  the same identifier in roman font denotes its pointset (i.e. $Z = \bigcup \mathscr Z$).
  As in Section~\ref{subsec:statement-capons-local}, we write
  \begin{equation*}
    \mathscr B_{E\times F}
    = \mathscr X_{E\times F}\times \mathscr Y_{E\times F},
    \qquad E\times F\in \bigcup_{I\times J\in \mathscr A} \mathscr E_I\times \mathscr F_J.
  \end{equation*}
  Firstly, we define the collections of dyadic intervals $\widetilde{\mathscr X}_{I\times J}$ by
  \begin{equation}\label{proof:lem:product:local-product:1}
    \widetilde{\mathscr X}_{I\times J}
    = \bigcup_{\substack{E\in \mathscr E_I\\F\in \mathscr F_J}} \mathscr X_{E\times F},
    \qquad I\times J\in \mathscr A.
  \end{equation}
  Secondly, we define the collections of dyadic intervals $\widetilde{\mathscr Y}_{I\times J}$ by
  \begin{equation}\label{proof:lem:product:local-product:2}
    \widetilde{\mathscr Y}_{I\times J}
    = \bigcup_{\substack{E\in \mathscr E_I\\F\in \mathscr F_J}} \mathscr Y_{E\times F},
    \qquad I\times J\in \mathscr A.
  \end{equation}
  Thirdly, observe that the following identity is true:
  \begin{equation}\label{proof:lem:product:local-product:3}
    \widetilde{\mathscr B}_{I\times J}
    = \widetilde{\mathscr X}_{I\times J} \times \widetilde{\mathscr Y}_{I\times J},
    \qquad I\times J\in \mathscr A.
  \end{equation}
  
  We now show \textref[P]{enu:p1} for $(\widetilde{\mathscr B}_{I\times J})$.
  Let $I_0\times J_0, I_1\times J_1 \in \mathscr A$ and assume that
  $\mathscr B_{I_0\times J_0} \cap \mathscr B_{I_1\times J_1} \neq \emptyset$.
  Then there exist dyadic intervals $E_0\in \mathscr E_{I_0}$, $E_1\in \mathscr E_{I_1}$ and
  $F_0\in \mathscr F_{I_0}$, $F_1\in \mathscr F_{I_1}$ such that
  $\mathscr B_{E_0\times F_0}\cap \mathscr B_{E_1\times F_1}\neq \emptyset$.
  Since $(\mathscr B_{E\times F})$ satisfies \textref[P]{enu:p1}, we infer $E_0=E_1$ and $F_0=F_1$.
  Thus, $\mathscr E_{I_0}\cap \mathscr E_{I_1} \neq \emptyset$ and
  $\mathscr F_{I_0}\cap \mathscr F_{I_1} \neq \emptyset$, which implies $I_0=I_1$ and $J_0=J_1$.
  Now, let $I\times J\in \mathscr A$ and assume that there are
  $K_0\times L_0, K_1\times L_1\in \widetilde{\mathscr B}_{I\times J}$ such that
  $K_0\times L_0\cap K_1\times L_1 \neq \emptyset$, i.e. 
  $K_0\cap K_1 \neq \emptyset$ and $L_0\cap L_1 \neq \emptyset$.
  Clearly, there exist $E_0, E_1\in \mathscr E_I$ and $F_0, F_1\in \mathscr F_I$ so that
  $K_i\in \mathscr X_{E_i\times F_i}$ as well as $L_i\in \mathscr Y_{E_i\times F_i}$, $i=0,1$.
  This implies $X_{E_0}\cap X_{E_1} \supset X_{E_0\times F_0} \cap X_{E_1\times F_1} \neq \emptyset$ and
  $Y_{F_0}\cap Y_{F_1}\supset Y_{E_0\times F_0} \cap Y_{E_1\times F_1} \neq \emptyset$.
  Hence, by~\textref[P]{enu:p2} for the sequence of collections $(\mathscr B_{E\times F})$, we obtain
  $E_0\cap E_1 \neq \emptyset$ and $F_0\cap F_1 \neq \emptyset$.
  Since each of the two collections $\mathscr E_I$ and $\mathscr F_J$ consists of pairwise disjoint
  dyadic intervals, we note $E_0=E_1$ and $F_0=F_1$.
  Thus, $K_0\times L_0\cap K_1\times L_1\neq \emptyset$ and
  $K_i\times L_i\in \mathscr B_{E_0\times F_0}$, $i=0,1$, so by~\textref[P]{enu:p1} we have that
  $K_0\times L_0 = K_1\times L_1$.

  Next, we prove that $(\widetilde{\mathscr B}_{I\times J})$ has property \textref[P]{enu:p2}.
  To this end, let $I_k\times J_k\in \mathscr A$, $k=0,1$.
  We now show that
  \begin{equation}\label{proof:lem:product:local-product:4}
    \widetilde X_{I_0}\cap \widetilde X_{I_1} \neq \emptyset
    \quad\text{implies}\quad
    I_0\cap I_1 \neq \emptyset.
  \end{equation}
  Let $\widetilde X_{I_0}\cap \widetilde X_{I_1} \neq \emptyset$.
  By~\eqref{proof:lem:product:local-product:1} and~\eqref{eq:symbols-2} we obtain
  \begin{equation}\label{proof:lem:product:local-product:5}
    \widetilde X_{I_k} = \bigcup_{E\in \mathscr E_{I_k}} X_E,
    \qquad k=0,1,
  \end{equation}
  thus we can find dyadic intervals $E_k\in \mathscr E_{I_k}$, $k=0,1$, such that
  $X_{E_0}\cap X_{E_1}\neq \emptyset$.
  But then, by~\textref[P]{enu:p2} for $X_{E_0}, X_{E_1}$, we have that
  $E_0\cap E_1 \neq \emptyset$, and therefore
  $\mathscr E_{I_0}\cap \mathscr E_{I_1} \neq \emptyset$.
  By~\textref[P]{enu:p2} for $\mathscr E_{I_0}, \mathscr E_{I_1}$ we obtain
  $I_0\cap I_1 \neq \emptyset$, which proves~\eqref{proof:lem:product:local-product:4}.
  Next, we prove
  \begin{equation}\label{proof:lem:product:local-product:6}
    X_{I_0} \subset X_{I_1}
    \quad \text{whenever $I_0\subset I_1$}.
  \end{equation}
  Let $I_0\subset I_1$.
  By~\textref[P]{enu:p2} for $E_{I_k}$, $k=0,1$, we have that $E_{I_0}\subset E_{I_1}$.
  By~\eqref{enu:lem:product:local-product:ii}, we obtain that for all $E_0\in \mathscr E_{I_0}$
  there is an $E_1\in \mathscr E_{I_1}$ such that $E_0 \subset E_1$.
  \textref[P]{enu:p2} for $X_{E_k}$, $k=0,1$, implies $X_{E_0}\subset X_{E_1}$.
  Thus, we obtain from~\eqref{proof:lem:product:local-product:5}
  that~\eqref{proof:lem:product:local-product:6} holds.
  The respective proof for $Y_{J_0}, Y_{J_1}$ is repeating the above proof for
  $X_{I_0}, X_{I_1}$, with $Y$ replacing $X$, $F$ replacing $E$ and $I$ replacing $J$.

  Next, we will prove~\textref[P]{enu:p3}.
  Again, since the proof for $X_I$ and $Y_J$ is completely analogous, we will
  prove~\textref[P]{enu:p3} only for $X_I$.
  Let $I\times J\in \mathscr A$.
  By~\eqref{proof:lem:product:local-product:5}, we have that
  \begin{equation*}
    \widetilde X_I = \bigcup_{E\in \mathscr E_I} X_E.
  \end{equation*}
  By~\eqref{enu:lem:product:local-product:i} for $\mathscr E_I$ and~\textref[P]{enu:p3} for $X_E$,
  $E\in \mathscr E_I$ and the above identity, we obtain
  \begin{equation*}
    C_X^{-1} |E_I|
    \leq C_X^{-1} \sum_{E\in \mathscr E_I} | E |
    \leq | \widetilde X_I | = \sum_{E\in \mathscr E_I} | X_E |
    \leq C_X \sum_{E\in \mathscr E_I} | E |
    = C_X |E_I|.
  \end{equation*}
  By~\textref[P]{enu:p3} we have that $C_X^{-1}|I| \leq |E_I|\leq C_X |I|$, which combined with the
  above estimate shows
  \begin{equation*}
    C_E^{-1}C_X^{-1} |I|
    \leq |\widetilde X_I|
    \leq C_EC_X |I|.
  \end{equation*}
  We note that~\textref[P]{enu:p3} holds with constants $C_EC_X$ and $C_FC_Y^2$, respectively.

  Finally, we will show that~\textref[P]{enu:p4} holds for $\widetilde {\mathscr B}_{I\times J}$.
  For brevity, we will only show the estimates concerning $X_{I_0}$.
  The estimates for $Y_{J_0}$ follow by replacing the proper characters in the proof given below.
  Let $I_0\times J_0, I\times J\in \mathscr A$ with $I_0\times J_0\subset I\times J$, and let
  $K\in \widetilde{\mathscr X}_{I\times J}$.
  By~\eqref{proof:lem:product:local-product:1}, there exist $E\in \mathscr E_I$ and
  $F\in \mathscr F_J$ such that $K\in \mathscr X_{E\times F}$.
  By~\eqref{proof:lem:product:local-product:5}, \eqref{enu:lem:product:local-product:i}
  and~\textref[P]{enu:p2} for $X_{E_0}$, we obtain that
  \begin{equation*}
    \frac{|K\cap \widetilde X_{I_0}|}{|K|}
    = \sum_{E_0\in \mathscr E_{I_0}}\frac{|K\cap X_{E_0}|}{|K|}
    \geq \sum_{\substack{E_0\in \mathscr E_{I_0}\\E_0\subset E}}\frac{|K\cap X_{E_0}|}{|K|}.
  \end{equation*}
  Using~\textref[P]{enu:p4} for $E_0\subset E$, $K\in \mathscr X_{E\times F}$, we obtain
  \begin{equation*}
    \frac{|K\cap \widetilde X_{I_0}|}{|K|}
    \geq C_X^{-1} \sum_{\substack{E_0\in \mathscr E_{I_0}\\E_0\subset E}}\frac{|X_{E_0}|}{|X_E|}
    \geq C_X^{-3} \sum_{\substack{E_0\in \mathscr E_{I_0}\\E_0\subset E}}\frac{|E_0|}{|E|},
  \end{equation*}
  where the latter estimate follows from~\textref[P]{enu:p3} for $X_{E_0}$ and $X_E$.
  Using the hypothesis~\eqref{enu:lem:product:local-product:i} yields
  \begin{equation*}
    \frac{|K\cap \widetilde X_{I_0}|}{|K|}
    \geq C_X^{-3} \frac{|E\cap E_{I_0}|}{|E|}.
  \end{equation*}
  Invoking~\textref[P]{enu:p4} for $\frac{|E\cap E_{I_0}|}{|E|}$ gives
  \begin{equation*}
    \frac{|K\cap \widetilde X_{I_0}|}{|K|}
    \geq C_E^{-1}C_X^{-3} \frac{|I_0|}{|I|}.
  \end{equation*}
  We note that~\textref[P]{enu:p4} holds with constants $C_EC_X^3$ and $C_FC_Y^3$, respectively.
\end{proof}

\section{Local results}\label{sec:local-results}

In this section we show how to almost--diagonalize operators on finite dimensional mixed norm
Hardy spaces and their duals, by building a block basis $(b_R^{(\varepsilon)})$ which satisfies
the local product conditions~\textref[P]{enu:p1}--\textref[P]{enu:p4}.
Moreover, if $T$ has large diagonal with respect to the bi--parameter Haar system, then $T$ has
large diagonal with respect to the block basis $(b_R^{(\varepsilon)})$.
This is achieved in Theorem~\ref{thm:quasi-diag}.

Combining Theorem~\ref{thm:quasi-diag} with Theorem~\ref{thm:projection} yields the local factorization
Theorem~\ref{thm:localized-factorization}, which asserts that the identity operator on a finite dimensional
mixed norm Hardy space (or its dual) factors through operators with large diagonal, in a larger,
finite dimensional mixed norm Hardy space (or its dual).

As a by--product of Theorem~\ref{thm:quasi-diag} and Theorem~\ref{thm:projection}, we obtain that the sequences of
finite dimensional mixed norm Hardy spaces $(H^p_n(H^q_n))_{n\in \mathbb N}$ and
$(H^p_n(H^q_n)^*)_{n\in \mathbb N}$ both have the property that projections almost annihilate finite
dimensional subspaces, see Definition~\ref{dfn:property-pafds} and Theorem~\ref{thm:projections-that-annihilate}.

\subsection{A combinatorial lemma in \bm{$H^p(H^q)$}}\label{subsec:combinatorial}\hfill

\noindent
The following Lemma~\ref{lem:comb-1} will be used as a quantifiable substitute for weak limits in the
proofs of the quantitative local results Theorem~\ref{thm:quasi-diag} and Theorem~\ref{thm:localized-factorization}.
Although the proof of Lemma~\ref{lem:comb-1} is merely a repetition of the proof given
in~\cite{lechner_mueller:2014} for the case $p=q=1$, one still has to check that it does in fact
work for the mixed norm Hardy spaces.
For that reason and for sake of completeness, we give the proof below.

\begin{lem}\label{lem:comb-1}
  Let $1\leq p,q < \infty$ and with the usual convention let $1 < p', q' \leq \infty$ denote the
  indices given by $\frac{1}{p}+\frac{1}{p'} = 1$ and $\frac{1}{q}+\frac{1}{q'} = 1$.
  Let $i \in \mathbb N$, $K_0, L_0 \in \dint$, and for all $0\leq j\leq i-1$ let
  $x_j \in H^p(H^q)^*$ and $y_j \in H^p(H^q)$ be such that
  \begin{equation}\label{eq:function-hypothesis}
    \sum_{j=0}^{i-1} \|x_j\|_{H^p(H^q)^*} \leq |K_0|^{1/p'} |L_0|^{1/q'}
    \quad\text{and}\quad
    \sum_{j=0}^{i-1} \|y_j\|_{H^p(H^q)} \leq |K_0|^{1/p} |L_0|^{1/q}.
  \end{equation}
  The local frequency weight $f_i$ is given by
  \begin{equation}\label{eq:local_frequency_weight}
    f_i(K\times L)
    = \sum_{j=0}^{i-1} |\langle x_j, h_{K\times L} \rangle| + |\langle h_{K\times L}, y_j \rangle|,
    \quad K\times L\in \drec.
  \end{equation}
  Given $\tau > 0$, $r \in \mathbb N_0$, we define the collections of dyadic intervals
  \begin{align*}
    \mathscr K(K_0\times L_0) &= \big\{ K\times L_0 : K \subset K_0,\ |K|\leq 2^{-r}|K_0|,\
      f_i(K\times L_0)
      \leq \tau |K\times L_0|
    \big\},\\
    \mathscr L(K_0\times L_0) &= \big\{ K_0\times L : L \subset L_0,\ |L|\leq 2^{-r}|L_0|,\
      f_i(K_0\times L)
      \leq \tau |K_0\times L|
    \big\}.
  \end{align*}
  For all integers $k,\ell$ the collections $\mathscr K_k(K_0\times L_0)$ and
  $\mathscr L_\ell(K_0\times L_0)$ are given by
  \begin{align*}
    \mathscr K_k(K_0\times L_0)
    &= \mathscr K(K_0\times L_0) \isect (\{K \in \dint : |K| = 2^{-k}|K_0|\}\times \dint),\\
    \mathscr L_\ell(K_0\times L_0)
    & = \mathscr L(K_0\times L_0) \isect (\dint\times \{L \in \dint : |L| = 2^{-\ell}|L_0|\}).
  \end{align*}
  Let $\rho > 0$.
  Then there exist integers $k,\ell$ with
  \begin{equation}\label{lem:comb-1:int}
    r \leq k,\ell \leq \bigg\lfloor \frac{4}{\rho^2\tau^2} \bigg\rfloor + r
  \end{equation}
  such that
  \begin{equation}\label{lem:comb-1:measure}
    | \mathscr K_k^*(K_0\times L_0) | \geq (1-\rho) |K_0\times L_0|
    \quad\text{and}\quad
    | \mathscr L_\ell^*(K_0\times L_0) | \geq (1-\rho) |K_0\times L_0|.
  \end{equation}
\end{lem}
Note that the $x$--component of the rectangles in $\mathscr K_k(K_0\times L_0)$ cover a set of
measure $\geq (1-\rho) |K_0|$ in $K_0$, and the $y$--component of the rectangles in
$\mathscr L_\ell(K_0\times L_0)$ cover a set of measure $\geq (1-\rho) |L_0|$ in $L_0$.
To be precise
\begin{equation}\label{lem:comb-1:measure:coordinatewise}
  \begin{aligned}
    \big| \bigcup \{ K : K\times L\in \mathscr K_k(K_0\times L_0) \} \big|
    & \geq (1-\rho) |K_0|,\\
    \big| \bigcup \{ L : K\times L\in \mathscr L_\ell(K_0\times L_0) \} \big|
    & \geq (1-\rho) |L_0|.
  \end{aligned}
\end{equation}
The estimate~\eqref{lem:comb-1:measure:coordinatewise} follows from the fact that all
$K\times L\in \mathscr K_k(K_0\times L_0)$ are such that $L=L_0$, and similarly, the
estimate~\eqref{lem:comb-1:measure:coordinatewise} for the collection
$\mathscr L_\ell(K_0\times L_0)$ follow from the fact that all
$K\times L\in \mathscr L_\ell(K_0\times L_0)$ are such that $K=K_0$.
See~Figure~\ref{fig:combinatorial-2} for a depiction of the collections $\mathscr K_k(K_0\times L_0)$ and
$\mathscr L_\ell(K_0\times L_0)$.
\begin{figure}[bt]
  \begin{center}
    \includegraphics[scale=0.25]{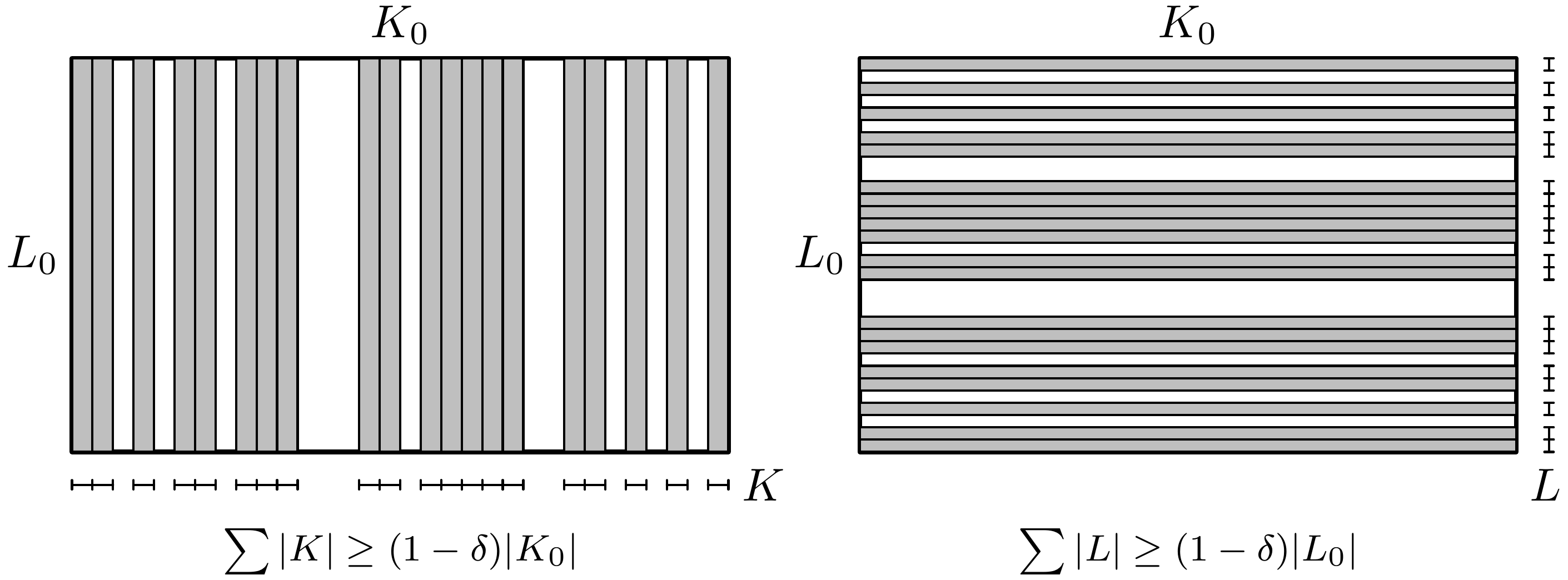}
  \end{center}
  \caption{The gray rectangles in the left picture form the collection
    $\mathscr K_k(K_0\times L_0)$, the gray rectangles to the right form
    $\mathscr L_\ell(K_0\times L_0)$.}
  \label{fig:combinatorial-2}
\end{figure}

\begin{proof}
  Define $\mathscr B = \{K\times L_0 : K\subset K_0\} \setminus \mathscr K(K_0\times L_0)$ and
  \begin{equation*}
    \mathscr B_k = \mathscr B \isect (\{K \in \dint : |K| = 2^{-k}|K_0|\}\times \dint),
  \end{equation*}
  for all $k\in \mathbb N$.
  Now let
  \begin{equation*}
    A = \bigg\lfloor \frac{4}{\rho^2\tau^2} \bigg\rfloor + r.
  \end{equation*}
  Since each $R\in \mathscr B(K_0\times L_0)$ has the form $K\times L_0$ for some
  $K\in \dint$, we have
  \begin{align*}
    \big\| \sum_{k=r}^A \sum_{R\in \mathscr B_k} \pm h_R \big\|_{H^p(H^q)}
    & = \big\| \sum_{k=r}^A \sum_{K : K\times L_0\in \mathscr B_k} \pm h_K \big\|_{H^p}
    \|h_{L_0}\|_{H^q},\\
    \big\| \sum_{k=r}^A \sum_{R\in \mathscr B_k} \pm h_R \big\|_{H^p(H^q)^*}
    & = \big\| \sum_{k=r}^A \sum_{K : K\times L_0\in \mathscr B_k} \pm h_K \big\|_{{H^p}^*}
    \|h_{L_0}\|_{{H^q}^*}.
  \end{align*}
  It is easily verified that $\|h_{L_0}\|_{H^q} = |L_0|^{1/q}$,
  $\|h_{L_0}\|_{{H^q}^*} = |L_0|^{1/q'}$, and that
  \begin{align*}
    \big\| \sum_{k=r}^A \sum_{K : K\times L_0\in \mathscr B_k} \pm h_K \big\|_{H^p}
    & \leq \sqrt{A-r+1}\, |K_0|^{1/p},\\
    \big\| \sum_{k=r}^A \sum_{K : K\times L_0\in \mathscr B_k} \pm h_K \big\|_{{H^p}^*}
    & \leq \sqrt{A-r+1}\, |K_0|^{1/p'}.
  \end{align*}
  Thus, we note the following estimates
  \begin{equation}\label{eq:lem:comb-1:1}
    \begin{aligned}
      \big\| \sum_{k=r}^A \sum_{R\in \mathscr B_k} \pm h_R \big\|_{H^p(H^q)}
      & \leq \sqrt{A-r+1}\, |K_0|^{1/p} |L_0|^{1/q},\\
      \big\| \sum_{k=r}^A \sum_{R\in \mathscr B_k} \pm h_R \big\|_{H^p(H^q)^*}
      & \leq \sqrt{A-r+1}\, |K_0|^{1/p'} |L_0|^{1/q'}.
    \end{aligned}
  \end{equation}

  By construction, $\mathscr B_k$ and $\mathscr K_k(K_0\times L_0)$ form a disjoint decomposition of
  $K_0\times L_0$.
  We will determine a collection $\mathscr K_k(K_0\times L_0)$ by showing that $\mathscr B_k^*$ is small
  enough for at least one value of $k$.
  Now assume the opposite, namely that
  \begin{equation*}
    | \mathscr B_k^* | \geq \rho |K_0\times L_0|,
    \qquad r \leq k \leq A.
  \end{equation*}
  Summing these estimates yields
  \begin{equation}\label{eq:lem:comb-1:2}
    \sum_{k=r}^A | \mathscr B_k^* | \geq (A-r+1)\, \rho\, |K_0\times L_0|,
  \end{equation}
  Observe that by definition of $\mathscr B$ and $\mathscr K$
  \begin{equation*}
    \tau\cdot \sum_{k=r}^A |\mathscr B_k^*|
    \leq \sum_{j=0}^{i-1} \sum_{k=r}^A \sum_{K\times L_0 \in \mathscr B_k}
      |\langle x_j, h_{K\times L_0} \rangle| + |\langle h_{K\times L_0}, y_j\rangle|.
  \end{equation*}
  Rewriting the right hand side in the following way
  \begin{equation*}
    \sum_{j=0}^{i-1} \Big|\Big\langle
      x_j, \sum_{k=r}^A \sum_{K\times L_0 \in \mathscr B_k} \pm h_{K\times L_0}
    \Big\rangle\Big|
    + \Big|\Big\langle
      \sum_{k=r}^A \sum_{K\times L_0 \in \mathscr B_k} \pm h_{K\times L_0}, y_j
    \Big\rangle\Big|,
  \end{equation*}
  and using~\eqref{eq:function-hypothesis} together with~\eqref{eq:lem:comb-1:1} yields
  \begin{equation*}
    \tau\cdot \sum_{k=r}^A |\mathscr B_k^*|
    \leq 2 \sqrt{A-r+1}\, |K_0\times  L_0|.
  \end{equation*}
  Combining the latter estimate with~\eqref{eq:lem:comb-1:2}, we obtain
  \begin{equation*}
    A \leq \frac{4}{\rho^2\tau^2} + r - 1,
  \end{equation*}
  which contradicts the definition of $A$.
  Thus we found $r\leq k\leq A$ so that
  \begin{equation*}
    | \mathscr K_k^*(K_0\times L_0) | \geq (1-\rho) |K_0\times L_0|,
  \end{equation*}
  see Figure~\ref{fig:combinatorial-2}.

  The same proof carried out in the other variable can be used the show the estimate for the
  collections $\mathscr L_\ell(K_0\times L_0)$.
\end{proof}

\subsection{Quantitative almost--diagonalization}\label{subsec:almost-diag}\hfill\\
\noindent
We show that any given operator on a finite dimensional mixed normed Hardy space or its dual can be almost--diagonalized by a block basis $(b_R^{(\varepsilon)})$ of the bi--parameter Haar system.
If moreover, the operator has large diagonal with respect to the bi--parameter Haar system, then it
has large diagonal with respect to $(b_R^{(\varepsilon)})$.
The block basis is such that it satisfies the local product
conditions~\textref[P]{enu:p1}--\textref[P]{enu:p4}.
We provide quantitative estimates on the number of block basis elements, which depends
(among other things) on the dimension of the Hardy space.

\begin{thm}\label{thm:quasi-diag}
  Let $1 \leq p,q < \infty$ and $\delta \geq 0$.
  For $m,n \in \mathbb N$ and $\Gamma, \eta > 0$ there exists an integer
  $N=N(m,n,\Gamma,\eta)$ so that the following holds:
  For any operator $T : H^p_m(H^q_N)\rightarrow H^p_m(H^q_N)$
  or $T : H^p_m(H^q_N)^*\rightarrow H^p_m(H^q_N)^*$
  with $\|T\|\leq \Gamma$ satisfying
  \begin{equation}\label{thm:quasi-diag:large}
    \langle h_R, T h_R \rangle \geq \delta |R|,
    \qquad\text{$R\in \dint^m\times \dint^N$
    },
  \end{equation}
  there exists a finite sequence of collections $(\mathscr B_R : R\in \dint^m\times\dint^n)$ and a
  sequence of signs $(\varepsilon_Q : Q\in \drec)$ defining a block basis of the Haar system
  $b_R^{(\varepsilon)}$ by
  \begin{equation*}
    b_R^{(\varepsilon)} = \sum_{Q \in \mathscr B_R} \varepsilon_Q h_Q,
    \qquad R\in \dint^m\times\dint^n,
  \end{equation*}
  so that the following conditions are satisfied:
  \begin{enumerate}[(i)]
  \item $\mathscr B_R \subset \dint^m\times\dint^N$,
    for all $R\in \dint^m\times\dint^n$.
    \label{enu:thm:quasi-diag-i}

  \item $(\mathscr B_R : R\in \dint^m\times\dint^n)$ satisfies local product
    conditions~\textref[P]{enu:p1}--\textref[P]{enu:p4} with constants
    $C_X = 1$ and $C_Y = 1+\eta$.
    \label{enu:thm:quasi-diag-ii}

  \item $(b_R^{(\varepsilon)} : R\in \dint^m\times\dint^n)$ almost--diagonalizes $T$ so that $T$ has
    large diagonal with respect to $(b_R^{(\varepsilon)} : R\in \dint^m\times\dint^n)$.
    To be more precise, we have the estimates
    \begin{subequations}\label{thm:quasi-diag-iii}
      \begin{align}
        \sum_{\substack{R' \in \dint^m\times\dint^n\\R'\neq R}}
        | \langle b_R^{(\varepsilon)}, T b_{R'}^{(\varepsilon)} \rangle |
        &\leq \eta \|b_R^{(\varepsilon)}\|_2^2,
        & R &\in \dint^m\times\dint^n,
        \label{thm:quasi-diag-iii:a}\\
        \langle b_R^{(\varepsilon)}, T b_R^{(\varepsilon)} \rangle
        &\geq \delta \|b_R^{(\varepsilon)}\|_2^2,
        & R &\in \dint^m\times\dint^n.
        \label{thm:quasi-diag-iii:b}
      \end{align}
    \end{subequations}
    \label{enu:thm:quasi-diag-iii}
  \end{enumerate}
\end{thm}
The proof of Theorem~\ref{thm:quasi-diag} relies on a modification of the construction
in~\cite{lechner_mueller:2014} for the sequence of collections of dyadic rectangles
$(\mathscr B_R)$, the combinatorial Lemma~\ref{lem:comb-1} in $H^p(H^q)$ to make the off--diagonal small,
and selecting signs $(\varepsilon_Q)$ for the block basis
$b_R^{(\varepsilon)} = \sum_{Q\in \mathscr B_R} \varepsilon_Q h_Q$ to keep the diagonal large.
See~\cite{andrew:1979} for the one--parameter setting.

\bfpar{Order relation \pmb{$\drless$} on \pmb{$\dint^m\times\dint^n$}}

The proof of Theorem~\ref{thm:quasi-diag} is by induction over the dyadic rectangles
$\dint^m\times \dint^n$, hence, we need to linearly order them.
To this end, let $\lesslex$ denote the lexicographic order on $\mathbb R^4$.
We define the linear order $\drless$ on $\dint^m\times \dint^n$ by
\begin{equation}\label{eq:ordering}
  I_0\times J_0 \drless I_1\times J_1
  \quad\text{if and only if}\quad
  (|J_1|,|I_1|,\inf I_0, \inf J_0) \lesslex (|J_0|,|I_0|,\inf I_1, \inf J_1),
\end{equation}
where $I_0\times J_0, I_1\times J_1\in \dint^m\times \dint^n$.
By $\drindex$ we denote the index function given by the following conditions:
The function
\begin{equation*}
  \drindex : \dint^m\times \dint^n\to \{k\in \mathbb Z : 0\leq k < (2^{m+1} -1) (2^{n+1}-1)\}
\end{equation*}
is bijective and satisfies
\begin{equation*}
  \drindex(R_0) < \drindex(R_1)
  \qquad\text{if and only if}\qquad
  R_0\drless R_1.
\end{equation*}
See Figure~\ref{fig:rectangle-order}.
\begin{figure}[bt]
  \begin{center}
    \includegraphics[scale=0.25]{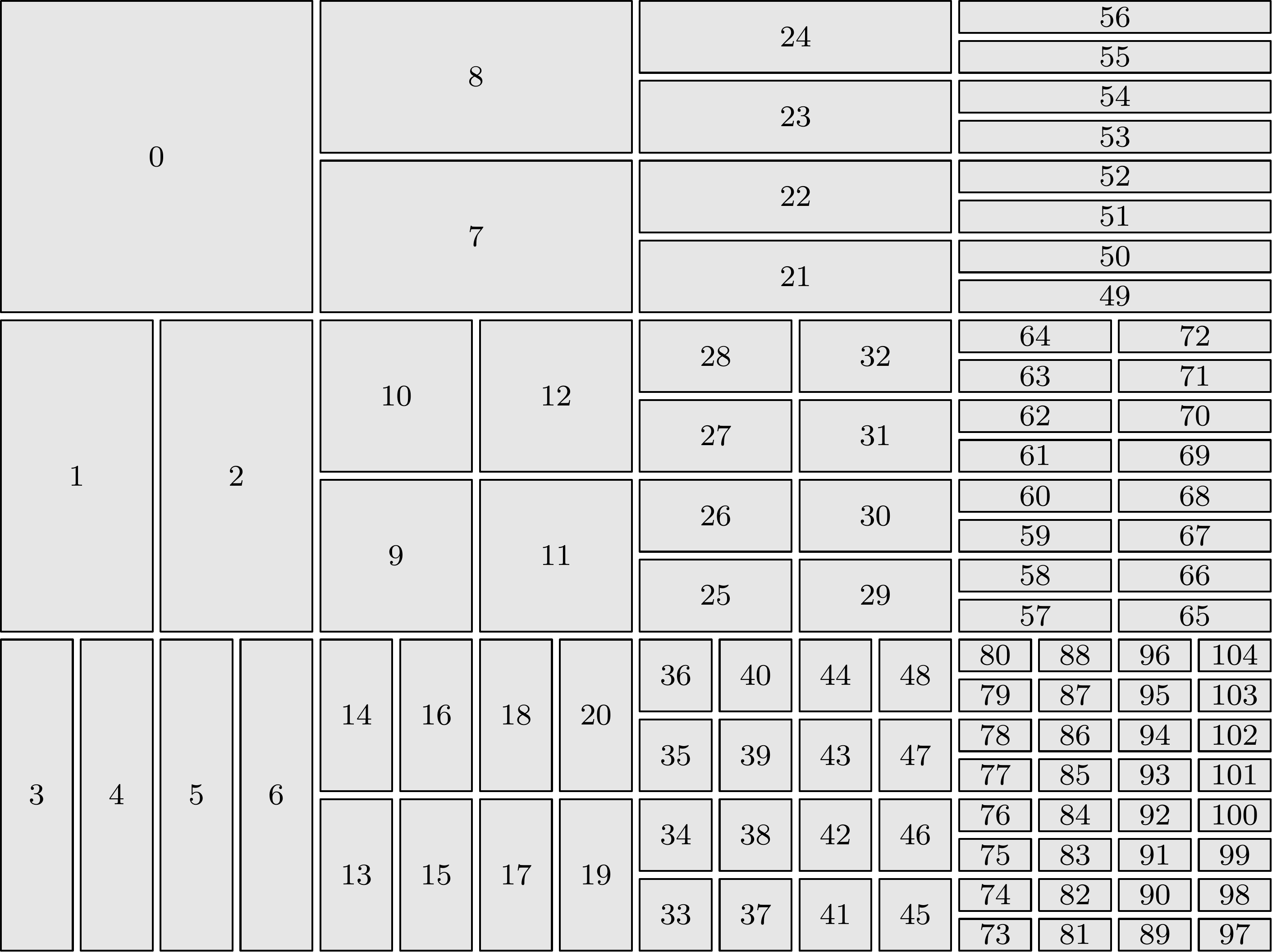}
  \end{center}
  \caption{Ordering of the 105 rectangles in $\dint^2\times\dint^3$.}
  \label{fig:rectangle-order}
\end{figure}

\begin{proof}[Proof of Theorem~\ref{thm:quasi-diag}]
  We only proof the case where $T : H^p_m(H^q_N)\rightarrow H^p_m(H^q_N)$.
  The proof for the case $T : H^p_m(H^q_N)^*\rightarrow H^p_m(H^q_N)^*$ is the same, but the roles of
  $T$ and $T^*$ are reversed.

  The proof is separated into the following steps:
  \begin{itemize}
  \item preparation,
  \item inductive construction of $b_{i_0}^{(\varepsilon)}$,
    \begin{itemize}
    \item construction of $\mathscr B_{i_0}$,
    \item selecting the signs $\varepsilon$,
    \end{itemize}
  \item verification that $\mathscr B_{i_0}$ satisfies the local product
    conditions~\textref[P]{enu:p1}--\textref[P]{enu:p4},
  \item verification that $b_{i_0}^{(\varepsilon)}$ almost--diagonalizes $T$.
  \end{itemize}

  \bfpar{Preparation}

  Let $1 \leq p,q < \infty$, $\delta \geq 0$, $m,n\in \mathbb N$ and $\Gamma,\eta > 0$.
  The number
  $N = N(m,n,\Gamma,\eta)$ will be determined in the course of the proof.
  Let $T : H^p_m(H^q_N)\to H^p_m(H^q_N)$ be an operator with $\|T\|\leq \Gamma$ such that
  \begin{equation}\label{eq:large-diagonal-0}
    \langle h_Q, T h_Q \rangle \geq \delta |Q|,
    \qquad Q\in \dint^m\times \dint^N.
  \end{equation}

  Given $Q\in \dint^m\times \dint^N$ we write
  \begin{subequations}\label{eq:decomp}
    \begin{equation}
      T h_Q = \alpha_Q h_Q + r_Q,
    \end{equation}
    where
    \begin{equation}
      \alpha_Q = \frac{\langle h_Q, T h_Q \rangle}{|Q|}
      \qquad\text{and}\qquad
      r_Q = \sum_{Q' : Q'\neq Q}
      \frac{\langle h_{Q'}, T h_Q \rangle}{|Q'|} h_{Q'}.
    \end{equation}
  \end{subequations}
  Note that for all $Q=K\times L\in \dint^m\times \dint^N$ we have the estimates
  \begin{equation}\label{eq:a-estimate}
    \delta \leq \alpha_Q \leq \|T\|
    \qquad\text{and}\qquad
    \|r_Q\|_{H^p(H^q)} \leq 2\|T\| |K|^{1/p} |L|^{1/q}.
  \end{equation}

  \bfpar{Inductive construction of \bm{$(b_R^{(\varepsilon)} : R\in \pmb{\dint}^m\times
      \pmb{\dint}^n)$}}

  For fixed $R\in \dint^m\times \dint^n$, the block basis element $b_R^{(\varepsilon)}$ will be
  determined by a collection of dyadic rectangles $\mathscr B_R\subset \dint^m\times \dint^N$ and
  signs $\varepsilon=(\varepsilon_Q : Q\in \dint^m\times \dint^N)$, and is of the following form:
  \begin{equation}\label{eq:block_basis}
    b_R^{(\varepsilon)} = \sum_{Q\in \mathscr B_R} \varepsilon_Q h_Q.
  \end{equation}

  From now on, we sytematically use the following rule:
  whenever $\drindex(R) = i$, we set
  \begin{equation*}
    \mathscr B_i = \mathscr B_R
    \qquad\text{and}\qquad
    b_i^{(\varepsilon)} = b_R^{(\varepsilon)}.
  \end{equation*}
  In the course of this proof we will construct the finite sequence of collections
  $(\mathscr B_R : R\in \dint^m\times \dint^n)$ and signs
  $\varepsilon = (\varepsilon_Q : Q\in \dint^m\times \dint^N)$ so that
  $(\mathscr B_R : R\in \dint^m\times \dint^n)$ satisfies the local product
  conditions~\textref[P]{enu:p1}--\textref[P]{enu:p4} with constants $C_X = 1$ and $C_Y = 1+\eta$,
  and that the block basis $(b_i^{(\varepsilon)})_{i\in \mathbb N_0}$ given
  by~\eqref{eq:block_basis} satisfies
  \begin{subequations}\label{eq:induction-properties}
    \begin{align}
      \sum_{j=0}^{i-1} |\langle T^*b_j^{(\varepsilon)}, b_i^{(\varepsilon)}\rangle|
      + |\langle b_i^{(\varepsilon)}, T b_j^{(\varepsilon)}\rangle|
      & \leq c(\eta') 4^{-i-1} \|b_i^{(\varepsilon)}\|_2^2,
      \label{eq:induction-properties:b}\\
      \langle b_i^{(\varepsilon)}, T b_i^{(\varepsilon)} \rangle
      & \geq \delta \|b_i^{(\varepsilon)}\|_2^2,
      \label{eq:induction-properties:c}
    \end{align}
  \end{subequations}
  for all $i\in \mathbb N_0$, where $c(\eta')\to 0$ as $\eta'\to 0$.
  We now choose $\eta' = \eta'(m,n,\Gamma,\eta) > 0$ so small that
  \begin{equation}\label{eq:eta'}
    (1-\eta')^{-1} \leq 1 + \eta
    \qquad\text{and}\qquad
    \eta' (1-\eta')^{-2} 4^{m+n}\Gamma \leq \eta.
  \end{equation}

  The induction begins by putting
  \begin{equation}\label{eq:induction-init}
    \mathscr B_0 = \{[0,1)\times [0,1)\}
    \qquad\text{and}\qquad
    b_0^{(\varepsilon)} = h_{[0,1)\times [0,1)}.
  \end{equation}
  Let $i_0\in \mathbb N$.
  At this stage we assume:
  \begin{itemize}
  \item There exist collections $\mathscr B_j = \mathscr B_{I\times J}$, for all
    $\drindex(I\times J) = j \leq i_0 -1$ of the the form
    \begin{equation*}
      \mathscr B_{I\times J}
      = \big\{ I\times L : L\in \mathscr Y_{I\times J}\big\},
    \end{equation*}
    where $\mathscr Y_{I\times J}$ is a finite subset of $\dint$.
    In the notation of Section~\ref{sec:clpc}, $\mathscr X_{I\times J} = \{I\}$.

  \item The collections $\mathscr B_j$, $0\leq j\leq i_0-1$, are such that
    $(\mathscr B_j)_{j=0}^{i_0-1}$ satisfies the local product
    conditions~\textref[P]{enu:p1}--\textref[P]{enu:p4}.

  \item The block basis elements $b_j^{(\varepsilon)}$, given by~\eqref{eq:block_basis},
    satisfy~\eqref{eq:induction-properties} for $0\leq j\leq i_0-1$.
  \end{itemize}
  Now, we turn to the construction of $\mathscr B_{i_0}$ and $\varepsilon_Q$,
  where $Q\in \mathscr B_{i_0}$.
  In the first step we find $\mathscr B_{i_0}$, and only then we will determine the signs
  $\varepsilon$.

  \bfpar{Construction of $\pmb{\mathscr B}\bm{_{i_0}}$}

  Let $I_0\times J_0\in \drec$ such that $\drindex(I_0\times J_0) = i_0$.
  At the beginning of the construction as well as at the end, we will distinguish between the two
  cases
  \begin{equation*}
    I_0 \neq [0,1)
    \qquad\text{and}\qquad
    I_0 = [0,1).
  \end{equation*}
  In both cases, we will use the combinatorial~Lemma~\ref{lem:comb-1}.

  If $I_0 \neq [0,1)$ we define the collection $\mathscr P_{I_0\times J_0}$ by
  \begin{subequations}\label{eq:case-1:past-indices}
    \begin{equation}\label{eq:case-1:past-indices:a}
      \mathscr P_{I_0\times J_0}
      = \{I\times J\in \dint^m\times\dint^n :
      I\times J \drless I_0\times J_0, I\neq I_0
      \},
    \end{equation}
    and if $I_0 = [0,1)$, we put
    \begin{equation}\label{eq:case-1:past-indices:b}
      \mathscr P_{[0,1)\times J_0}
      = \{I\times J\in \dint^m\times\dint^n :
      I\times J \drless [0,1)\times J_0, |J| > |J_0|
      \}.
    \end{equation}
  \end{subequations}
  In both cases, we now define $\mathbb A_{I_0\times J_0}$ by
  \begin{equation}\label{eq:intersection-classes}
    \mathbb A_{I_0\times J_0} = \big\{
    \{I\times J'\in \mathscr P_{I_0\times J_0} : |J'| = |J|\} :
    I\times J\in \mathscr P_{I_0\times J_0}
    \big\}.
  \end{equation}
  Before we proceed with the proof, we make a few remarks.
  \begin{itemize}
  \item For all $J\in \dint^n$ with $|J|=|J_0|$ holds that
    $\mathscr P_{I_0\times J} = \mathscr P_{I_0\times J_0}$, and hence
    $\mathbb A_{I_0\times J} = \mathbb A_{I_0\times J_0}$.

  \item If $I\times J\in\mathscr A\in \mathbb A_{I_0\times J_0}$, then
    \begin{equation}\label{eq:case-2:partition:2}
      \mathscr A = \{ I\times J' : J'\in \dint,\ |J'|=|J| \}
    \end{equation}
    see~\eqref{eq:ordering} and~\eqref{eq:intersection-classes}.

  \item The collection $\mathbb A_{I_0\times J_0}$ is a partition of
    $\mathscr P_{I_0\times J_0}$, i.e.
    \begin{equation*}
      \bigcup \mathbb A_{I_0\times J_0}
      = \mathscr P_{I_0\times J_0}
      \qquad\text{and}\qquad
      \mathscr A\cap \mathscr A' = \emptyset
    \end{equation*}
    for all $\mathscr A, \mathscr A'\in \mathbb A_{I_0\times J_0}$ with $\mathscr A\neq \mathscr A'$.

  \item The collections $\mathscr Y_{I\times J}$ have already been constructed for all
    $I\times J\in \mathscr P_{I_0\times J_0}$.

  \item Let $\mathscr A\in \mathbb A_{I_0\times J_0}$ and
    $I\times J, I\times J'\in \mathscr A$, then
    \begin{equation*}
      Y_{I\times J}\cap Y_{I\times J'} = \emptyset,
      \qquad\text{if $J\neq J'$}.
    \end{equation*}
  \end{itemize}

  For each $\mathscr A\in \mathbb A_{I_0\times J_0}$, let $W(\mathscr A)$ denote the set given by
  \begin{equation}\label{eq:case-2:local-intersection-set}
    W(\mathscr A)
    = \bigcup_{I\times J\in \mathscr A} Y_{I\times J}.
  \end{equation}
  Note that for each $\mathscr A\in \mathbb A_{I_0\times J_0}$ the set $W(\mathscr A)$ is an almost
  cover of the unit interval.
  Now we define $W_{I_0\times J_0}$ by intersecting all the $W(\mathscr A)$:
  \begin{equation}\label{eq:case-2:intersection-set}
    W_{I_0\times J_0}
    = \bigcap_{\mathscr A\in \mathbb A_{I_0\times J_0}} W(\mathscr A).
  \end{equation}
  We will cover the set $W_{I_0\times J_0}$ with smaller intervals than we have previously used in our
  construction.
  To this end let
  \begin{equation}\label{eq:case-2:gamma}
    \gamma_{i_0} = \gamma_{I_0\times J_0} = \frac{1}{2}\min\{|L| :
    \exists I\times J \drless I_0\times J_0,\ \mathscr Y_{I\times J}\ni L
    \}
  \end{equation}
  and define the high frequency cover $\mathscr W_{I_0\times J_0}$ of $W_{I_0\times J_0}$ by \begin{equation}\label{eq:case-2:high-frequency-cover:1}
    \mathscr W_{I_0\times J_0}
    = \{L_0\in \dint : |L_0| = \gamma_{I_0\times J_0},\ L_0\subset W_{I_0\times J_0}\}.
  \end{equation}
  Note the following identity:
  \begin{equation}\label{eq:case-2:high-frequency-cover:2}
    \Union \mathscr W_{I_0\times J_0} = W_{I_0\times J_0}.
  \end{equation}
  To each of the rectangles $I_0\times L_0$, where $L_0\in \mathscr W_{I_0\times J_0}$, we
  will now prepare to apply Lemma~\ref{lem:comb-1}, so that $I_0$ will remain intact, and $L_0$ will
  be almost covered with high frequencies $L$.
  To this end let
  \begin{equation}\label{eq:case-2:building_block_size}
    \beta_{i_0}
    = \beta_{I_0\times J_0}
    = \min\{|I_0\times L_0| : L_0\in \mathscr W_{I_0\times J_0}\}
  \end{equation}
  and define for all $0 \leq j \leq i_0-1$
  \begin{align}
    x_j &:= \frac{\beta_{i_0}}{\Gamma i_0 \|b_j^{(\varepsilon)}\|_{H^p(H^q)^*}} T^*b_j^{(\varepsilon)},
    &y_j &:= \frac{\beta_{i_0}}{\Gamma i_0 \|b_j^{(\varepsilon)}\|_{H^p(H^q)}} T b_j^{(\varepsilon)}.
    \label{eq:proof:case-2:functions}
  \end{align}
  Recall that $\|T\| \leq \Gamma$, hence
  \begin{equation*}
    \sum_{j=0}^{i_0-1} \|x_j\|_{H^p(H^q)^*} \leq \beta_{i_0}
    \qquad\text{and}\qquad
    \sum_{j=0}^{i_0-1} \|y_j\|_{_{H^p(H^q)}} \leq \beta_{i_0}.
  \end{equation*}
  The local frequency weight $f_{i_0}$ is given by
  \begin{equation}\label{eq:case-2:local_frequency_weight}
    f_{i_0}(Q)
    = \sum_{j=0}^{i_0-1} |\langle x_j, h_Q \rangle| + |\langle y_j, h_Q\rangle|,
    \qquad Q\in \drec,
  \end{equation}
  and the constant $\tau_{i_0}$ by
  \begin{equation}\label{eq:proof:case-2:tau}
    \tau_{i_0}
    = \tau_{I_0\times J_0}
    = \frac{\eta' \beta_{i_0}}{i_0 4^{i_0+1}}.
  \end{equation}
  For each $L_0\in \mathscr W_{I_0\times J_0}$ we put
  \begin{equation}\label{eq:case-1:local_buidling_blocks}
    \mathscr L(I_0\times L_0)
    = \{I_0\times L :
    L\in \dint,\ L\subsetneq L_0,\ f_{i_0}(I_0\times L) \leq \tau_{i_0} |I_0\times L|
    \}.
  \end{equation}
  Finally, we define the constant $\rho_{i_0}$ by
  \begin{equation}\label{eq:proof:case-1:rho}
    \rho_{i_0}
    = \rho_ {I_0\times J_0}
    = \frac{\eta'}{4^{i_0}}.
  \end{equation}
  Since $\beta_{i_0} \leq |I_0|^{1/{p'}} |L_0|^{1/{q'}}$ and
  $\beta_{i_0} \leq |I_0|^{1/p} |L_0|^{1/q}$,
  for all $L_0\in \mathscr W_{I_0\times J_0}$, Lemma~\ref{lem:comb-1} yields an integer
  $\ell = \ell(I_0\times L_0)$ with
  \begin{equation}\label{eq:proof:case-1:k}
    1 \leq \ell \leq \bigg\lfloor \frac{4}{\rho_{i_0}^2\tau_{i_0}^2} \bigg\rfloor + 1,
  \end{equation}
  such that the collection $\mathscr L_\ell(I_0\times L_0)$ given by
  \begin{equation*}
    \mathscr L_\ell(I_0\times L_0)
    = \{I_0\times L\in \mathscr L(I_0\times L_0) : |L| = 2^{-\ell}|L_0|\}
  \end{equation*}
  satisfies the estimate
  \begin{equation*}
    (1-\rho_{i_0}) |I_0\times L_0|
    \leq \Big| \bigcup\{ Q : Q\in \mathscr L_\ell(I_0\times L_0) \} \Big|
    \leq |I_0\times L_0|.
  \end{equation*}
  Now, we take the union over all $L_0\in \mathscr W_{I_0\times J_0}$ to obtain
  \begin{equation}\label{eq:proof:case-1:Z}
    \mathscr Z_{I_0\times J_0}
    = \Union\big\{ \mathscr L_\ell(I_0\times L_0) :
    L_0\in \mathscr W_{I_0\times J_0}
    \big\}.
  \end{equation}
  Once again, we emphasize that $\ell=\ell(I_0\times L_0)$ in the above formula.
  Let $Z_{I_0\times J_0}$ denote the pointset of $\mathscr Z_{I_0\times J_0}$, i.e.
  \begin{equation*}
    Z_{I_0\times J_0} = \bigcup \mathscr Z_{I_0\times J_0},
  \end{equation*}
  then for all $L_0\in \mathscr W_{I_0\times J_0}$ we have the estimates
  \begin{equation}\label{eq:case-1:local-measure-estimate}
    (1-\rho_{i_0}) |I_0\times L_0|
    \leq \Big| Z_{I_0\times J_0}\cap (I_0\times L_0) \Big|
    \leq |I_0\times L_0|.
  \end{equation}
  We want to point out that $Q\in \mathscr Z_{I_0\times J_0}$ implies $Q=I_0\times L$, for
  some $L\in \dint$.
  There exists a unique $L_0\in \mathscr W_{I_0\times J_0}$ such that $L_0\supset L$, and therefore
  $|L| = 2^{-\ell}$, where $\ell = \ell(I_0\times L_0)$.

  \noindent
  \begin{minipage}[H]{.6\textwidth}
    \bfpar{Case~1: \bm{$I_0 \neq [0,1)$}.}
    Here, we know that $2^{-m} \leq |I_0| \leq 1/2$.
    Let $\widetilde I_0$ be the dyadic predecessor of $I_0$, then
    $\mathscr B_{\widetilde I_0\times J_0}$ has already been defined (see~\eqref{eq:ordering}).
    The block basis indexed by the dark gray rectangles has already been constructed.
    Here, we determine the block basis for the light gray rectangles.
    The white ones will be treated later.
  \end{minipage}
  \begin{minipage}[H]{.4\textwidth}
    \begin{center}
      \includegraphics[scale=0.1]{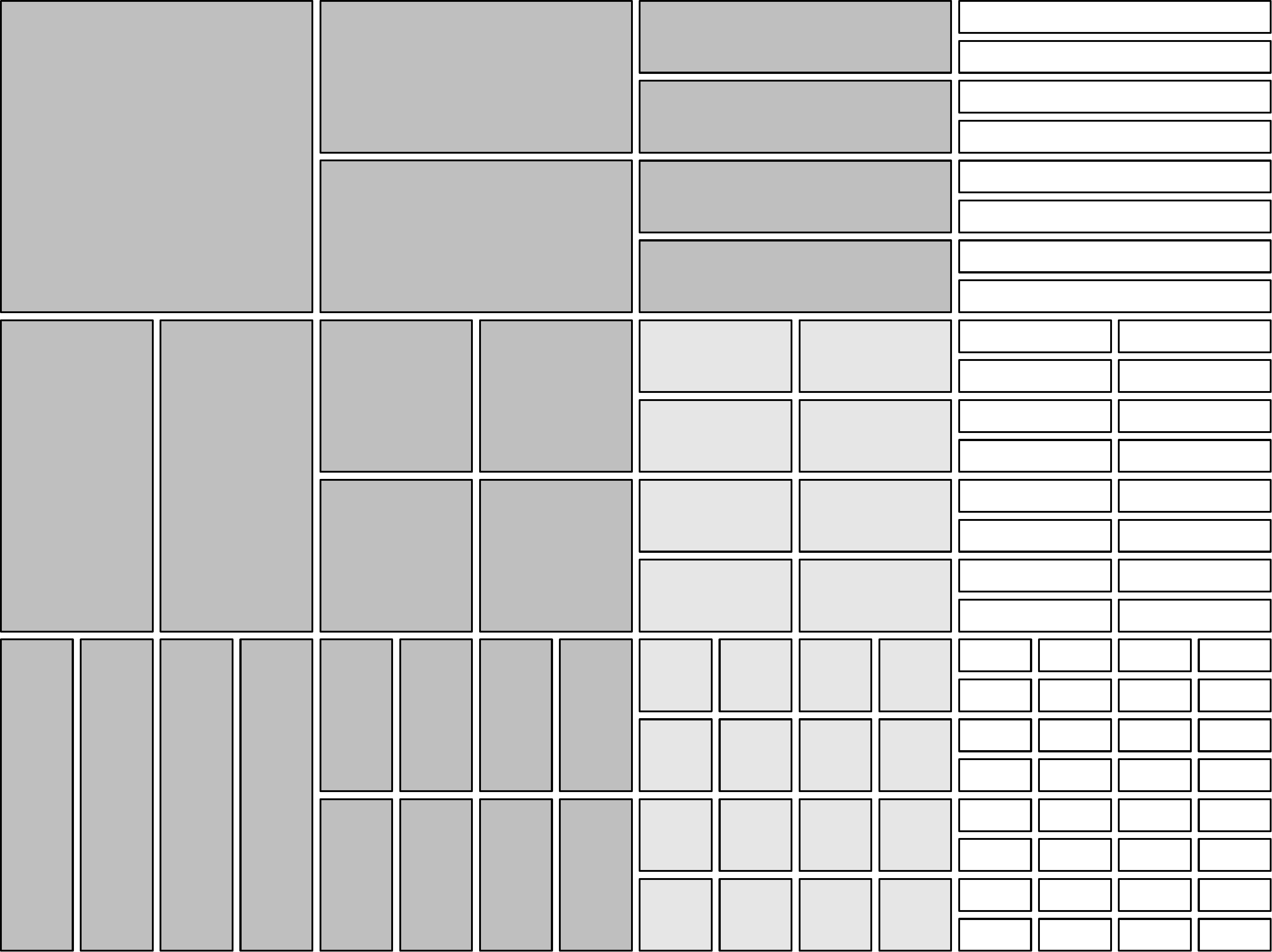}
    \end{center}
  \end{minipage}
  \noindent
  In this case we put
  \begin{equation}\label{eq:rectangles-case_1}
    \mathscr B_{I_0\times J_0}
    = \{ I_0\times L\in \mathscr Z_{I_0\times J_0} :
    I_0\times L \subset B_{\widetilde I_0\times J_0}
    \}.
  \end{equation}
  see Figure~\ref{fig:building_blocks-1}.
  \begin{figure}[H]
    \begin{center}
      \includegraphics[scale=0.25]{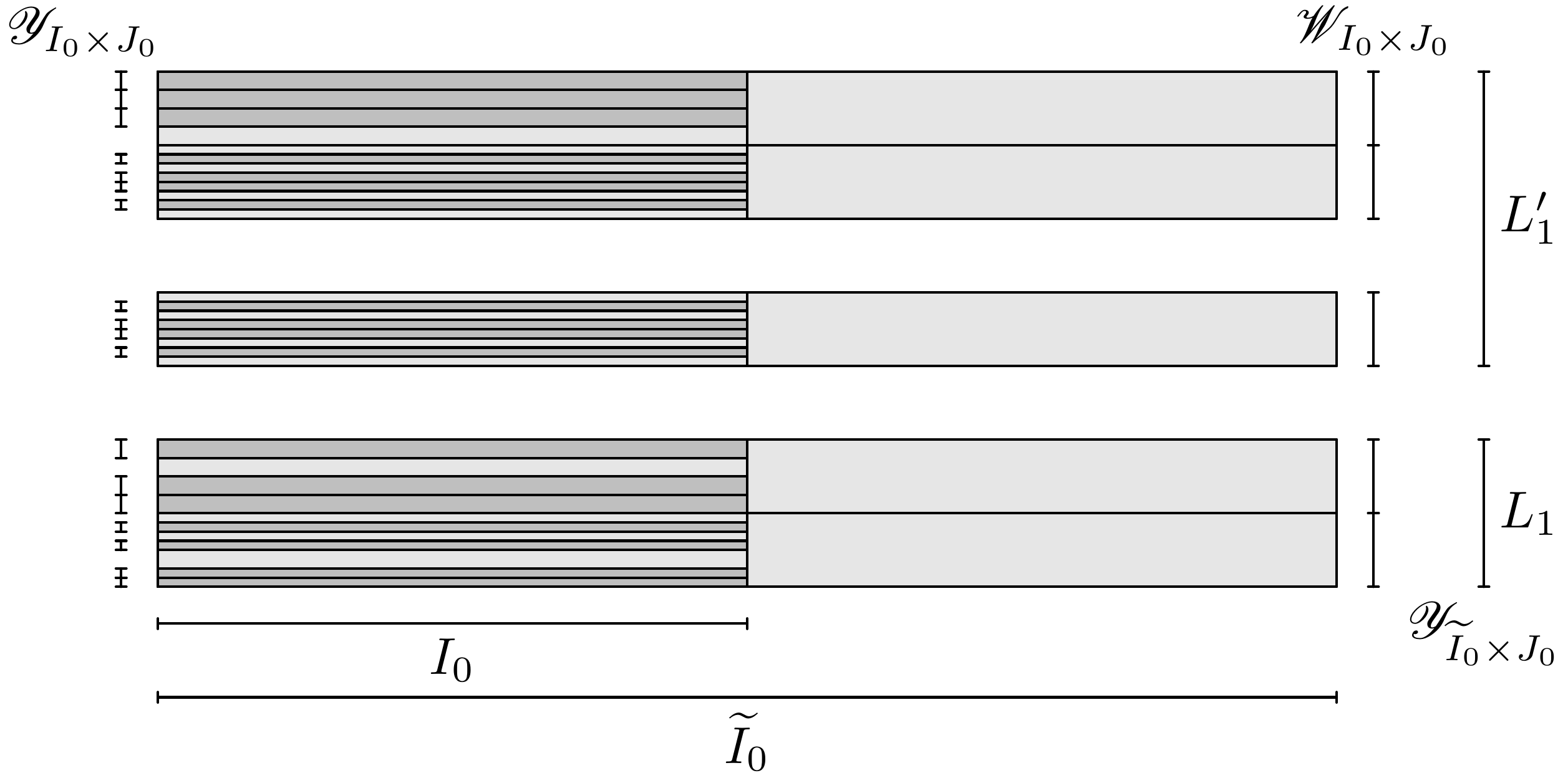}
    \end{center}
    \caption{In this figure, the dyadic interval $I_0$ is the left half of $\widetilde I_0$, and
      the collection $\mathscr Y_{\widetilde I_0\times J_0}$ consists of the dyadic intervals $L_1$ and
      $L_1'$.
      Immediately to the left of $L_1$, $L_1'$ are the unlabeled dyadic intervals of the collection
      $\mathscr W_{I_0\times J_0}$.
      Note that all the intervals in $\mathscr W_{I_0\times J_0}$ have the same length,
      see~\eqref{eq:case-2:high-frequency-cover:1}.
      The pointset $W_{I_0\times J_0}$ of $\mathscr W_{I_0\times J_0}$ almost covers each of the
      intervals $L_1$, $L_1'$, see~\eqref{eq:intersection-classes}, \eqref{eq:case-2:partition:2},
      \eqref{eq:case-2:local-intersection-set}, \eqref{eq:case-2:intersection-set}
      and~\eqref{eq:case-2:high-frequency-cover:2}.
      The light gray rectangles are formed by the product of $\widetilde I_0$ with the intervals in
      $\mathscr W_{I_0\times J_0}$.
      The collection of dyadic intervals $\mathscr Y_{I_0\times J_0}$ is given by the unlabeled
      vertical intervals to the very left of the figure are determined by the combinatorial
      Lemma~\ref{lem:comb-1}.
      The pointset $Y_{I_0\times J_0}$ of $\mathscr Y_{I_0\times J_0}$ almost covers each interval in
      $\mathscr W_{I_0\times J_0}$, and therefore each interval in
      $\mathscr Y_{\widetilde I_0\times J_0} = \{L_1, L_1'\}$,
      see~\eqref{eq:case-1:local_buidling_blocks}, \eqref{eq:proof:case-1:Z}
      \eqref{eq:case-1:local-measure-estimate} and~\eqref{eq:rectangles-case_1}.
      The collection $\mathscr B_{I_0\times J_0}$ consists of the dark gray rectangles and is given
      by $\{I_0\times L : L\in \mathscr Y_{I_0\times J_0}\}$.
    }
    \label{fig:building_blocks-1}
  \end{figure}
  By~\eqref{eq:rectangles-case_1} we obtain that
  \begin{equation}\label{eq:case-1:rectangle-form:final}
    K\times L\in \mathscr B_{I_0\times J_0}
    \quad\text{implies}\quad
    K = I_0,
  \end{equation}
  thus, \eqref{eq:case-1:local-measure-estimate}, \eqref{eq:rectangles-case_1}
  and~\eqref{eq:case-1:rectangle-form:final} yield
  \begin{subequations}\label{eq:case-1:local-measure-estimate:final}
    \begin{equation}\label{eq:case-1:local-measure-estimate:final:a}
      \mathscr X_{I_0\times J_0} = \{I_0\}
    \end{equation}
    and
    \begin{equation}\label{eq:case-1:local-measure-estimate:final:b}
      (1-\rho_{I_0\times J_0}) |L|
      \leq | Y_{I_0\times J_0}\cap L |
      \leq |L|,
    \end{equation}
  \end{subequations}
  for all $L\in \mathscr Y_{I\times J}$ with $L\cap Y_{I_0\times J_0}\neq \emptyset$, where
  $I\times J\in \mathscr P_{I_0\times J_0}$ is maximal with respect to the ordering~$\drless$
  (in which case $I\neq I_0$, $J=J_0$ and we have $L\cap Y_{I_0\times J_0}\neq \emptyset$ for all
  $L\in \mathscr Y_{I\times J}$).

  \noindent
  \begin{minipage}[H]{.6\textwidth}
    \bfpar{Case~2: \bm{$I_0 = [0,1)$}}
    In this case we know that $\mathscr B_{I\times \widetilde J_0}$ has already been constructed for
    all $I\in \dint^m$ (see~\eqref{eq:ordering}); those are the dark gray rectangles in the third
    column.
    Here, we determine the block basis for the light gray rectangles.
    The white ones will be treated later.
  \end{minipage}
  \begin{minipage}[H]{.4\textwidth}
    \begin{center}
      \includegraphics[scale=0.1]{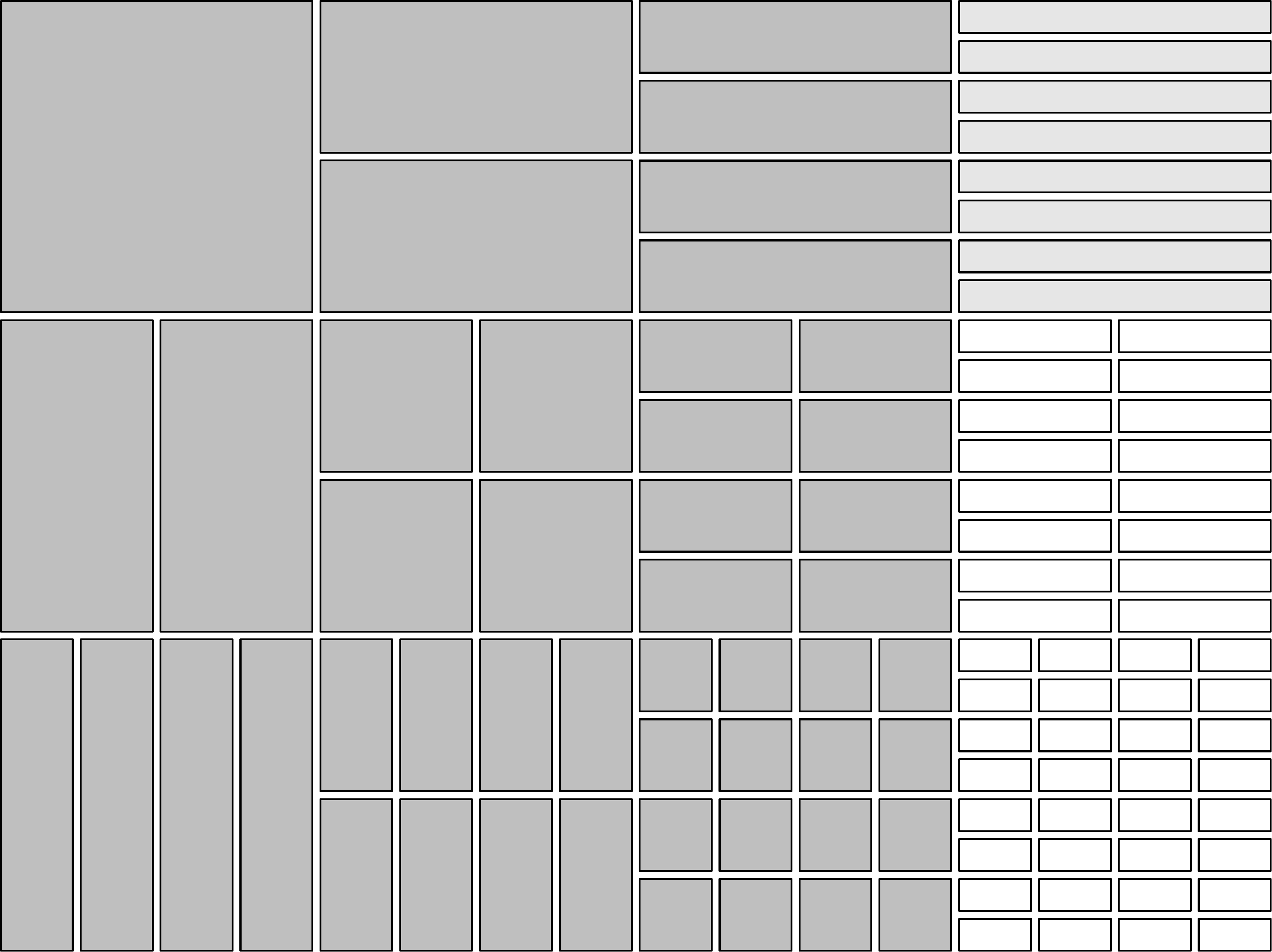}
    \end{center}
  \end{minipage}
  \noindent
  Here, we define the sets
  \begin{equation*}
    B_{[0,1)\times \widetilde J_0}^\ell
    = \Union_{[0,1)\times L_0\in \mathscr B_{[0,1)\times \widetilde J_0}} [0,1)\times L_0^\ell
  \end{equation*}
  and
  \begin{equation*}
    B_{[0,1)\times \widetilde J_0}^r
    = \Union_{[0,1)\times L_0\in \mathscr B_{[0,1)\times \widetilde J_0}} [0,1)\times L_0^r.
  \end{equation*}
  \begin{subequations}\label{eq:rectangles-case_2a}
    If $J_0$ is the \emph{left half} of $\widetilde J_0$ we put
    \begin{equation}\label{eq:rectangles-case_2a:left}
      \mathscr B_{[0,1)\times J_0}
      = \{[0,1)\times L\in \mathscr Z_{[0,1)\times J_0} :
      [0,1)\times L\subset B_{[0,1)\times \widetilde J_0}^\ell
      \}.
    \end{equation}
    If $J_0$ is the \emph{right half} of $\widetilde J_0$ we put
    \begin{equation}\label{eq:rectangles-case_2a:right}
      \mathscr B_{[0,1)\times J_0}
      = \{[0,1)\times L\in \mathscr Z_{[0,1)\times J_0} :
      [0,1)\times L\subset B_{[0,1)\times \widetilde J_0}^r
      \},
    \end{equation}
  \end{subequations}
  see Figure~\ref{fig:building_blocks-2}.
  \begin{figure}[H]
    \begin{center}
      \includegraphics[scale=0.25]{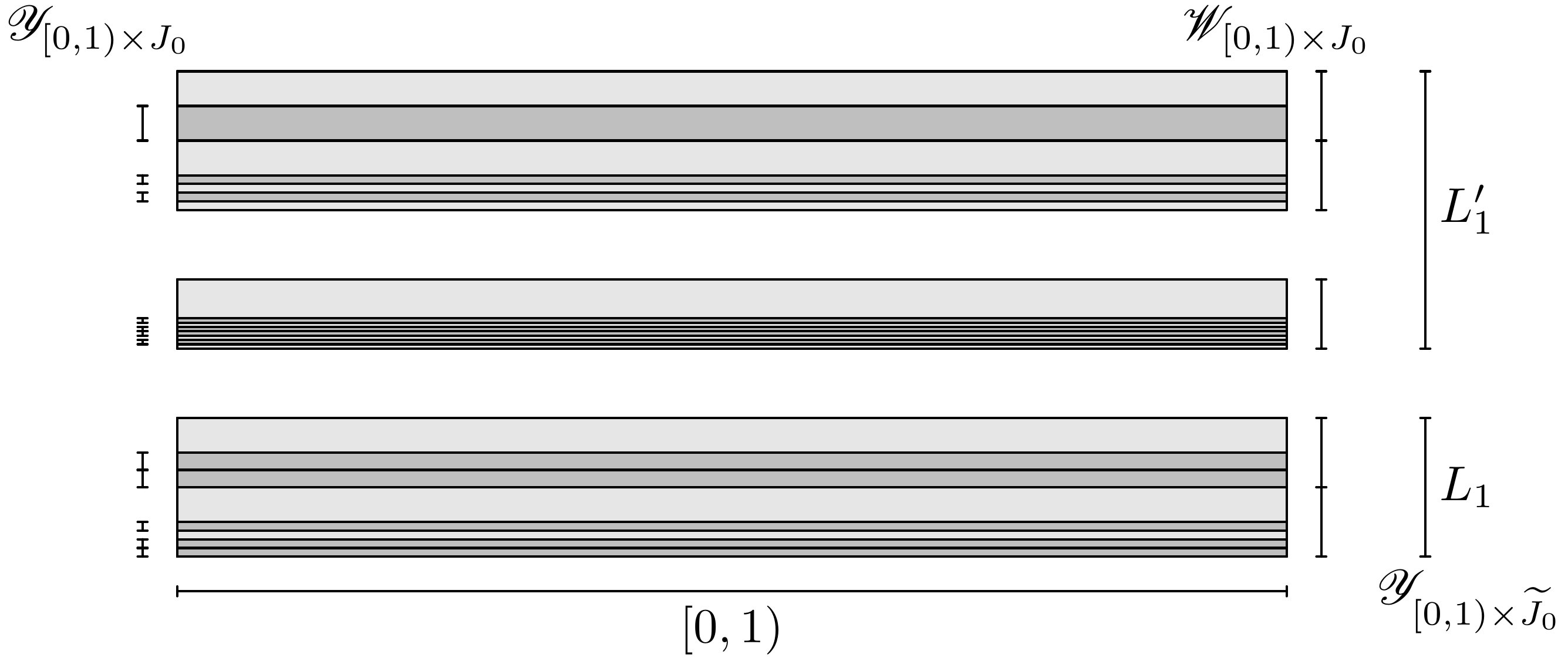}
    \end{center}
    \caption{In this figure, $J_0$ is the left half of $\widetilde J_0$.
      The collection of dyadic intervals $\mathscr Y_{[0,1)\times \widetilde J_0} = \{L_1,L_1'\}$
      are the $y$--components of the collection $\mathscr B_{[0,1)\times \widetilde J_0}$.
      To the left is the collection $\mathscr W_{[0,1)\times J_0}$, whose pointset
      $W_{[0,1)\times J_0}$ almost covers $L_1$ and $L_1'$.
      The light gray dyadic rectangles are determined by the products $[0,1)\times L_0$, where
      $L_0\in \mathscr W_{[0,1)\times J_0}$.
      The dark gray rectangles are the output of the combinatorial Lemma~\ref{lem:comb-1}, and are all
      contained in the left halves in one of the intervals $L_0\in \mathscr W_{[0,1)\times J_0}$
      (since $J_0$ is the left half of $\widetilde J_0$, see~\eqref{eq:rectangles-case_2a}).
      The dark gray rectangles form the collection $\mathscr B_{[0,1)\times J_0}$.
      The collection $\mathscr Y_{[0,1)\times J_0}$ is the collection of $y$--components of
      $\mathscr B_{[0,1)\times J_0}$.
    }
    \label{fig:building_blocks-2}
  \end{figure}
  In this case, we have by~\eqref{eq:rectangles-case_2a} that
  \begin{equation}\label{eq:case-2:rectangle-form:final}
    K\times L\in \mathscr B_{[0,1)\times J_0}
    \quad\text{implies}\quad
    K = [0,1),
  \end{equation}
  thus, \eqref{eq:case-1:local-measure-estimate}, \eqref{eq:rectangles-case_2a}
  and~\eqref{eq:case-2:rectangle-form:final} yield
  \begin{subequations}\label{eq:case-2:local-measure-estimate:final}
    \begin{equation}\label{eq:case-2:local-measure-estimate:final:a}
      \mathscr X_{[0,1)\times J_0} = \{[0,1)\}
    \end{equation}
    and
    \begin{equation}\label{eq:case-2:local-measure-estimate:final:b}
      \frac{1}{2}(1-2\rho_{[0,1)\times J_0}) |L|
      \leq | Y_{[0,1)\times J_0}\cap L |
      \leq |L|,
    \end{equation}
  \end{subequations}
  for all $L\in \mathscr Y_{I\times J}$ with $L\cap Y_{I_0\times J_0}\neq \emptyset$, where
  $I\times J\in \mathscr P_{I_0\times J_0}$ is maximal with respect to the ordering~$\drless$ (in
  which case $I=[1-2^{-m},1)$, $J = \widetilde J_0$ and $L\cap Y_{I_0\times J_0}\neq \emptyset$ for
  all $L\in \mathscr Y_{I\times J}$).

  Recall that  $i_0 = \drindex(I_0\times J_0)$, and that
  $\mathscr B_{i_0} = \mathscr B_{I_0\times J_0}$ as well as
  $b_{i_0}^{(\varepsilon)} = b_{I_0\times J_0}^{(\varepsilon)}$.
  By the definition of $\mathscr L(I_0\times L_0)$ (see~\eqref{eq:case-1:local_buidling_blocks}) and
  the definition of $\mathscr B_{i_0}$
  (see~\eqref{eq:rectangles-case_1} and~\eqref{eq:rectangles-case_2a}), we note that in both cases
  \begin{equation}\label{eq:case-1:almost-diagonal}
    f_{i_0}(Q) \leq \tau_{i_0} |Q|,
    \qquad Q\in \mathscr B_{i_0}.
  \end{equation}

  \bfpar{Selecting the signs $\bm{\varepsilon}$}

  In any of the above cases~\eqref{eq:rectangles-case_1} and~\eqref{eq:rectangles-case_2a}, we
  define the following function.
  For any choice of signs $\varepsilon_Q \in \{-1,+1\}$, $Q\in \mathscr B_{i_0}$ put
  \begin{equation}\label{eq:block-basis-candidate}
    b_{i_0}^{(\varepsilon)} = \sum_{Q\in \mathscr B_{i_0}} \varepsilon_Q h_Q.
  \end{equation}
  We will now select signs $(\varepsilon_Q : Q\in \mathscr B_{i_0})$ such that
  \begin{equation*}
    \langle b_{i_0}^{(\varepsilon)}, T b_{i_0}^{(\varepsilon)} \rangle
    \geq \delta \|b_{i_0}^{(\varepsilon)}\|_2^2.
  \end{equation*}

  To this end observe that~\eqref{eq:decomp} and~\eqref{eq:block-basis-candidate} yields
  \begin{equation*}
    \langle b_{i_0}^{(\varepsilon)}, T b_{i_0}^{(\varepsilon)} \rangle
    = \sum_{Q\in \mathscr B_{i_0}} \alpha_Q |Q|
    + \langle b_{i_0}^{(\varepsilon)}, R_m^{(\varepsilon)} \rangle,
  \end{equation*}
  where
  \begin{equation*}
    R_m^{(\varepsilon)} = \sum_{Q\in \mathscr B_{i_0}} \varepsilon_Q r_Q.
  \end{equation*}
  Thus, by~\eqref{eq:a-estimate} we get
  \begin{equation*}
    \langle b_{i_0}^{(\varepsilon)}, T b_{i_0}^{(\varepsilon)} \rangle
    \geq \delta \|b_{i_0}^{(\varepsilon)}\|_2^2
    + \langle b_{i_0}^{(\varepsilon)}, R_m^{(\varepsilon)} \rangle,
  \end{equation*}
  Let $\cond_\varepsilon$ denote the average over all possible choices of signs
  $(\varepsilon_Q : Q\in \mathscr B_{i_0})$.
  Using $\langle h_Q, r_Q \rangle = 0$, we obtain
  \begin{equation*}
    \langle b_{i_0}^{(\varepsilon)}, R_m^{(\varepsilon)} \rangle
    = \sum \varepsilon_{Q_0} \varepsilon_{Q_1}
    \langle h_{Q_0}, r_{Q_1} \rangle,
  \end{equation*}
  where the sum is taken over all $Q_0, Q_1\in \mathscr B_{i_0}$ with
  $Q_0\neq Q_1$.
  Taking the expectation on the right hand side we obtain,
  \begin{equation*}
    \cond_\varepsilon \langle b_{i_0}^{(\varepsilon)}, R_m^{(\varepsilon)} \rangle = 0.
  \end{equation*}
  This gives
  \begin{equation}\label{eq:diagonal_estimate-0}
    \cond_\varepsilon \langle b_{i_0}^{(\varepsilon)}, T b_{i_0}^{(\varepsilon)} \rangle
    \geq \delta \cond_\varepsilon \|b_{i_0}^{(\varepsilon)}\|_2^2,
  \end{equation}
  hence there exists at least one $\varepsilon$ such that
  \begin{equation}\label{eq:diagonal_estimate-1}
    \langle b_{i_0}^{(\varepsilon)}, T b_{i_0}^{(\varepsilon)} \rangle
    \geq \delta \|b_{i_0}^{(\varepsilon)}\|_2^2.
  \end{equation}

  \bfpar{$(\pmb{\mathscr B}\bm{_i : i\leq i_0})$ satisfies the local product conditions
    \textref[P]{enu:p1}--\textref[P]{enu:p4}.}

  It should be clear from the definition of $\mathscr F_m$ in each step, that
  $(\mathscr B_i : i\leq i_0)$ satisfies~\textref[P]{enu:p1}.
  Since $\mathscr X_{I\times J} = \{I\}$ for all $I\times J\in \dint^m\times \dint^n$,
  \textref[P]{enu:p2}--\textref[P]{enu:p4} is satisfied with $C_X = 1$.
  Recalling the definition of $Y_J$ (see~\eqref{eq:symbols-2}), and that the new
  $y$--components are obtained by intersecting all the supports from the previous steps
  (see~\eqref{eq:case-2:local-intersection-set}, \eqref{eq:case-2:intersection-set},
  \eqref{eq:case-2:high-frequency-cover:1}, \eqref{eq:case-2:high-frequency-cover:2},
  \eqref{eq:case-1:local_buidling_blocks}, \eqref{eq:proof:case-1:Z}, \eqref{eq:rectangles-case_1},
  and~\eqref{eq:rectangles-case_2a}) we observe that
  \begin{equation*}
    Y_J = \bigcup_{I\in \dint^m} Y_{I\times J} = Y_{[0,1)\times J},
    \qquad J\in \dint^n.
  \end{equation*}
  By considering~\eqref{eq:rectangles-case_2a} together with the above identity, it should be clear
  that $(Y_J : J\in \dint^n)$ satisfies~\textref[P]{enu:p2} and $|Y_J| \leq |J|$, $J\in \dint^n$.
  The remaining measure estimates~\textref[P]{enu:p3} and~\textref[P]{enu:p4} follow by induction
  from~\eqref{eq:case-1:local-measure-estimate:final}
  and~\eqref{eq:case-2:local-measure-estimate:final}.

  Now, let $I_0\times J_0, I\times J\in \dint^m\times \dint^n$ with
  $I_0\times J_0 \subsetneq I\times J$ and let $L\in \mathscr Y_{I\times J}$.
  From~\eqref{eq:case-1:local-measure-estimate:final}
  and~\eqref{eq:case-2:local-measure-estimate:final} follows immediately that $K_0=I_0$ and $K=I$.
  Since the $\rho_i$ are a geometric sequence (see~\eqref{eq:proof:case-1:rho}), we obtain by
  induction that
  \begin{equation}\label{eq:local-measure-estimate:Y}
    \frac{|J_0|}{|J|} ( 1 - \eta' ) |L|
    \leq |L\cap Y_{I_0\times J_0}|
    \leq \frac{|J_0|}{|J|} |L|.
  \end{equation}
  We remark that~\eqref{eq:local-measure-estimate:Y} implies \textref[P]{enu:p3}
  and~\textref[P]{enu:p4} with $C_Y = C_Y(\eta') = ( 1 - \eta')^{-1}$.
  To summarize, we showed that $(\mathscr B_R : R\in \dint^m \times \dint^n)$ satisfies the local
  product conditions~\textref[P]{enu:p1}--\textref[P]{enu:p4} with constants $C_X = 1$ and
  $C_Y = (1-\eta')^{-1}$.

  \bfpar{The block basis $\bm{(b_i^{(\varepsilon)} : i\leq i_0)}$ almost--diagonalizes $\bm{T}$.}

  First, recall that the constant $\tau_i$ was given by $\tau_i = \frac{\eta' \beta_i}{i 4^{i+1}}$
  (see~\eqref{eq:case-2:building_block_size} and~\eqref{eq:proof:case-2:tau}).
  With that in mind, we collect the estimates~\eqref{eq:case-1:almost-diagonal},
  (see~\eqref{eq:case-2:local_frequency_weight} for the definition of $f_{i_0}$)
  and the mixed norm estimates in Theorem~\ref{thm:projection} to obtain
  \begin{equation*}
    \sum_{j = 0}^{i-1} | \langle T^* b_j^{(\varepsilon)}, h_Q \rangle |
    + | \langle h_Q, T b_j^{(\varepsilon)} \rangle |
    \leq  \eta' (1-\eta')^{-1} 2^{m+n}\Gamma 4^{-i-1} |Q|,
  \end{equation*}
  for all $i$ and $Q\in \mathscr B_i$.
  From the latter estimate and the definition of $b_i^{(\varepsilon)}$
  (see~\eqref{eq:block-basis-candidate}) we obtain by summing over $Q\in \mathscr B_i$
  \begin{equation}\label{eq:almost-eigenvectors:0}
    \sum_{j = 0}^{i-1} | \langle T^* b_j^{(\varepsilon)}, b_i^{(\varepsilon)} \rangle |
    + | \langle b_i^{(\varepsilon)}, T b_j^{(\varepsilon)} \rangle |
    \leq  \eta' (1-\eta')^{-1} 2^{m+n}\Gamma 4^{-i-1},
    \qquad i \geq 1.
  \end{equation}
  From the first term in the sum of~\eqref{eq:almost-eigenvectors:0} follows the estimate
  \begin{equation*}
    |\langle b_i^{(\varepsilon)}, T b_j^{(\varepsilon)} \rangle|
    \leq  \eta' (1-\eta')^{-1} 2^{m+n}\Gamma 4^{-j-1},
    \qquad j > i \geq 0.
  \end{equation*}
  Summing over all those $j$ we obtain
  \begin{equation*}
    \sum_{j\geq i+1}
    |\langle b_i^{(\varepsilon)}, T b_j^{(\varepsilon)} \rangle|
    \leq  \eta' (1-\eta')^{-1} 2^{m+n}\Gamma 4^{-i-1},
    \qquad i \geq 0.
  \end{equation*}
  Combining the latter estimate with~\eqref{eq:almost-eigenvectors:0} and using that
  $\|b_i\|_2^2\geq (1-\eta')2^{-m-n}$
  (see~\eqref{eq:local-product-condition-i} and recall that $b_i = b_R$, whenever $i=\drindex(R)$,
  and that here $R\in \dint^m \times \dint^n$) yields
  \begin{equation}\label{eq:almost-eigenvectors:1}
    \sum_{\substack{j : j \neq i}} | \langle b_i^{(\varepsilon)}, Tb_j^{(\varepsilon)} \rangle |
    \leq \eta' (1-\eta')^{-2} 4^{m+n}\Gamma \|b_i^{(\varepsilon)}\|_2^2,
    \qquad i \geq 0.
  \end{equation}
  We remark that~\eqref{eq:almost-eigenvectors:1} and~\eqref{eq:eta'} together
  with~\eqref{eq:diagonal_estimate-1} proves~\eqref{thm:quasi-diag-iii}.

  Finally, observe that~\eqref{enu:thm:quasi-diag-i} of Theorem~\ref{thm:quasi-diag} holds true by
  observing that all the constants in the proof depend only on $m,n,\Gamma$ and $\eta$
  (see~\eqref{eq:eta'}, \eqref{eq:case-2:gamma},
  \eqref{eq:case-2:building_block_size}, \eqref{eq:proof:case-2:tau},
  \eqref{eq:case-1:local_buidling_blocks}, \eqref{eq:proof:case-1:rho}
  and~\eqref{eq:proof:case-1:k}).
\end{proof}

\begin{rem}\label{rem:thm:quasi-diag}
  Let $(e_K)$ and $(f_L)$ denote block bases of the one parameter Haar system satisfying the
  hypothesis of the reiteration Lemma~\ref{lem:product:local-product}, that is $(e_K\otimes f_L)$
  satisfies the local product conditions~\textref[P]{enu:p1}--\textref[P]{enu:p4}, and additional
  regularity assumptions (see
  Lemma~\ref{lem:product:local-product}~\eqref{enu:lem:product:local-product:i}
  and~\eqref{enu:lem:product:local-product:ii}).

  We remark that we could repeat the proof of Theorem~\ref{thm:quasi-diag} with $h_{K\times L}$ replaced
  by $\widetilde h_{K\times L} = e_K\otimes f_L$, $K\times L\in \dint^m\times \dint^N$.
  Due to the reiteration Lemma~\ref{lem:product:local-product}, we would arrive at the same conclusion.
  To be more precise: if we replace~\eqref{thm:quasi-diag:large} in Theorem~\ref{thm:quasi-diag} by
  \begin{equation}\label{thm:quasi-diag:large*}
    \langle \widetilde h_Q, T \widetilde h_Q \rangle
    \geq \delta \|\widetilde h_Q\|_2^2,
    \qquad Q\in \dint^m\times \dint^N,
  \end{equation}
  then all the conclusions \eqref{enu:thm:quasi-diag-i}--\eqref{enu:thm:quasi-diag-iii} of
  Theorem~\ref{thm:quasi-diag} remain valid for the block basis
  $(\widetilde b_R^{(\varepsilon)} : R\in \dint^m\times\dint^n)$ of the bi--parameter Haar system
  given by
  \begin{equation*}
    \widetilde b_R^{(\varepsilon)}
    = \sum_{Q \in \mathscr B_R} \varepsilon_Q \widetilde h_Q,
    \qquad R\in \dint^m\times\dint^n.
  \end{equation*}
\end{rem}

\subsection{Projections that almost annihilate finite dimensional subspaces}\label{subsec:projections-that-annihilate}\hfill

\noindent
In the proof of the main result Theorem~\ref{thm:factorization}, we will use the almost--diagonalization result Theorem~\ref{thm:quasi-diag}.
Additionally, we will need the following variation of Theorem~\ref{thm:quasi-diag}.
\begin{thm}\label{thm:projections-that-annihilate}
  Let $1\leq p,q < \infty$, $m,n,d \in \mathbb N$ and $\eta > 0$.
  Then there exists an integer
  $N=N(m,n,d,\eta)$ so that for any $d$--dimensional subspace $F \subset H_m^p(H_N^q)$
  (respectively $F \subset H_m^p(H_N^q)^*$)
  there exists a block basis $(b_R : R\in \dint^m\times\dint^n)$ satisfying the following
  conditions:
  \begin{enumerate}[(i)]
  \item $\mathscr B_R \subset \dint^m\times\dint^N$,
    for all $R \in \dint^m\times\dint^n$.
    \label{enu:thm:projections-that-annihilate-i}

  \item For every finite sequence of scalars $(a_R : R \in \dint^m\times\dint^n)$ we have that
    \begin{equation}\label{eq:thm:projections-that-annihilate:ii}
      (1+\eta)^{-1} \Big\| \sum_{R\in \dint^m\times\dint^n} a_R h_R \Big\|
      \leq \Big\| \sum_{R\in \dint^m\times\dint^n} a_R b_R \Big\|
      \leq (1+\eta) \Big\| \sum_{R\in \dint^m\times\dint^n} a_R h_R \Big\|.
    \end{equation}
    The above norms are either all the norm of $H^p(H^q)$, or they are all the norm of
    $H^p(H^q)^*$.
    \label{enu:thm:projections-that-annihilate-ii}

  \item The orthogonal projection $Q : H_m^p(H_N^q) \rightarrow H_m^p(H_N^q)$
    (respectively $Q : H_m^p(H_N^q)^* \rightarrow H_m^p(H_N^q)^*$)
    given by
    \begin{equation*}
      Q f = \sum_{R\in \dint^m\times\dint^n}
      \frac{\langle f, b_R \rangle}{\|b_R\|_2^2}\,
      b_R
    \end{equation*}
    satisfies the estimates
    \begin{equation}\label{eq:thm:projections-that-annihilate:iii}
      \begin{aligned}
        \|Q f\| &\leq (1+\eta) \|f\|,
        &f &\in H_m^p(H_N^q)\ \text{(respectively $f \in H_m^p(H_N^q)^*$)},\\
        \| Q f \| &\leq \eta\, \|f\|,  &f &\in F.
      \end{aligned}
    \end{equation}
    The above norms are either all the norm of $H^p(H^q)$, or they are all the norm of
    $H^p(H^q)^*$.
    \label{enu:thm:projections-that-annihilate-iii}
  \end{enumerate}
\end{thm}

\begin{proof}
  The proof of Theorem~\ref{thm:projections-that-annihilate} is a repetition of the
  almost-diagonalization argument in the proof of Theorem~\ref{thm:quasi-diag}, where the
  combinatorial Lemma~\ref{lem:comb-1} is used with the following frequency weight $f$ in each
  step.
  Given a finite $\eta/2$-net $(y_j)_{j=1}^k$ of the unit ball in $F$ define the local frequency
  weight $f$ by
  \begin{equation*}
    f(R) = \sum_{j=1}^k |\langle h_R, y_j\rangle|,
    \qquad R\in \drec.
  \end{equation*}
  Since we do not need a large diagonal in this particular instance, we choose all signs
  $\varepsilon_Q=1$.
  The bi-parameter case is analogous to the one parameter case, which is described in detail
  in~\cite[290--291]{mueller:2005}.
\end{proof}

\begin{rem}\label{rem:thm:projections-that-annihilate}
  In view of Remark~\ref{rem:thm:projection} and Remark~\ref{rem:thm:quasi-diag}, it is clear that we could have
  replaced the bi--parameter Haar system $(h_K\otimes h_L)$ by the tensor product
  $(e_K\otimes f_L)$, where $(e_K)$ and $(f_L)$ denote block bases of the one parameter Haar system,
  such that $(e_K\otimes f_L)$ satisfies~\textref[P]{enu:p1}--\textref[P]{enu:p4} as well as some
  additional regularity assumptions (see
  Lemma~\ref{lem:product:local-product}~\eqref{enu:lem:product:local-product:i}
  and~\eqref{enu:lem:product:local-product:ii}).
  Hence, the conclusions
  \eqref{enu:thm:projections-that-annihilate-i}--\eqref{enu:thm:projections-that-annihilate-iii}
  of Theorem~\ref{thm:projections-that-annihilate} are true for the block basis
  $(\widetilde b_R^{(\varepsilon)} : R\in \dint^m\times\dint^n)$ given by
  \begin{equation*}
    \widetilde b_R^{(\varepsilon)}
    = \sum_{K\times L \in \mathscr B_R} \varepsilon_{K\times L} e_K\otimes f_L,
    \qquad R\in \dint^m\times\dint^n.
  \end{equation*}
\end{rem}

\subsection{Local factorization}\label{subsec:local-factorization}

\noindent
Here, we state our local factorization result Theorem~\ref{thm:localized-factorization}, which follows by a
standard argument from the projection Theorem~\ref{thm:projection} and the almost--diagonalization result
Theorem~\ref{thm:quasi-diag}.
For sake of completeness and since we need to keep track of our constants, we repeat the proof
pattern in~\cite{lechner_mueller:2014}.
For the one--parameter analogue of this proof, we refer to~\cite[Chapter 5.2]{mueller:2005}.
\begin{thm}\label{thm:localized-factorization}
  Let $1 \leq p,q < \infty$ and $\delta > 0$.
  For $m,n \in \mathbb N$ and $\Gamma,\eta > 0$ there exists an integer
  $N=N(\delta,m,n,\Gamma,\eta)$ so that the following holds:
  For any operator $T : H^p_m(H^q_N)\rightarrow H^p_m(H^q_N)$
  (respectively $T : H^p_m(H^q_N)^*\rightarrow H^p_m(H^q_N)^*$)
  with $\|T\|\leq \Gamma$ satisfying
  \begin{equation*}
    |\langle h_R, T h_R \rangle| \geq \delta |R|,
    \qquad\text{$R\in \dint^m\times \dint^N$,
    }
  \end{equation*}
  the identity $\Id$ on $H^p_m(H^q_n)$ (respectively $H^p_m(H^q_n)^*$) well factors through $T$.
  To be more precise, there exist bounded linear operators
  $E : H^p_m(H^q_n)\to H^p_m(H^q_N)$ and $P : H^p_m(H^q_N)\to H^p_m(H^q_n)$
  (respectively $E : H^p_m(H^q_n)^*\to H^p_m(H^q_N)^*$ and $P : H^p_m(H^q_N)^*\to H^p_m(H^q_n)^*$)
  such that the diagram
  \begin{equation*}
    \vcxymatrix{H^p_m(H^q_n) \ar[r]^{\Id} \ar[d]_E & H^p_m(H^q_n)\\
      H^p_m(H^q_N) \ar[r]_T & H^p_m(H^q_N) \ar[u]_P}
    \qquad\text{respectively}\qquad
    \vcxymatrix{H^p_m(H^q_n)^* \ar[r]^{\Id} \ar[d]_E & H^p_m(H^q_n)^*\\
      H^p_m(H^q_N)^* \ar[r]_T & H^p_m(H^q_N)^* \ar[u]_P}
  \end{equation*}
  is commutative, and the operators $E$ and $P$ can be chosen so that
  $\|E\|\|P\| \leq (1+\eta)/\delta$.
\end{thm}
Note that for $T = \delta \Id_n$, we have $\|E\|\|P\|=1/\delta$.

The proof of Theorem~\ref{thm:localized-factorization} relies on Theorem~\ref{thm:quasi-diag}, which builds a
block basis $(b_R^{(\varepsilon)})$ of the bi--parameter Haar system that almost--diagonalizes the
operator $T$ while simultaneously maintaining the large diagonal:
$\langle b_R^{(\varepsilon)}, T b_R^{(\varepsilon)}\rangle \geq \delta|R|$,
$R\in \dint^m\times \dint^N$.
Moreover, $(b_R^{(\varepsilon)})$ satisfies the local product conditions, see Section~\ref{sec:clpc},
which implies that $(b_R^{(\varepsilon)})$ is equivalent to the bi--parameter Haar system in
$H^p(H^q)$ and $H^p(H^q)^*$, and that the orthogonal projection onto $(b_R^{(\varepsilon)})$ is
bounded on $H^p(H^q)$ and on $H^p(H^q)^*$.

\begin{proof}[Proof of Theorem~\ref{thm:localized-factorization}]
  We will only prove the case where $T : H^p_m(H^q_N)\rightarrow H^p_m(H^q_N)$, since the other case
  is completely analogous.
  But before we begin with the actual proof, observe that we can assume that
  \begin{equation*}
    \langle h_R, T h_R \rangle \geq \delta |R|,
    \qquad R\in \dint^m\times \dint^N.
  \end{equation*}
  Indeed, define $\mathcal M : H^p_m(H^q_N)\rightarrow H^p_m(H^q_N)$ as the linear extension of
  $\mathcal M h_R = \sign(\langle h_R, T h_R \rangle)$, $R\in \dint^m\times\dint^N$.  Note that
  $\mathcal M$ is a norm $1$ operator and $\langle h_R, T \mathcal M h_R \rangle \geq \delta |R|$.

  Now let $1\leq p,q < \infty$, $\delta > 0$, $m,n\in \mathbb N$ and $\gamma,\eta > 0$ be fixed.
  Let $\eta' > 0$ be a small constant satisfying the estimates
  \begin{equation}\label{eq:proof:eta'}
    \eta'\frac{mn}{(1+\eta')\delta} < 1,
    \qquad\text{and}\qquad
    \frac{(1+\eta')^{3k}}{\delta - \eta'\frac{mn}{1+\eta'}}\leq \frac{1+\eta}{\delta}.
  \end{equation}
  By Theorem~\ref{thm:quasi-diag}, we can find an integer $N = N(m,n,\Gamma,\eta')$ so that for any
  operator $T : H^p_m(H^q_N)\rightarrow H^p_m(H^q_N)$ with $\|T\|\leq \Gamma$, there exist
  collections $(\mathscr B_R : R\in \dint^m\times\dint^n)$ and signs $(\varepsilon_Q : Q\in \drec)$
  defining a block basis of the Haar system $(b_R^{(\varepsilon)} : R\in \dint^m\times \dint^n)$ by
  \begin{equation*}
    b_R^{(\varepsilon)} = \sum_{Q \in \mathscr B_R} \varepsilon_Q h_Q,
    \qquad R\in \dint^m\times\dint^n,
  \end{equation*}
  so that the following conditions are satisfied:
  \begin{enumerate}[(a)]
  \item $\mathscr B_R \subset \dint^m\times\dint^N$,
    for all $R\in \dint^m\times\dint^n$.
    \label{enu:proof:factorization-a}

  \item $(\mathscr B_R : R\in \dint^m\times\dint^n)$ satisfies the local product conditions
    (see Section~\ref{sec:clpc}) with constants $C_X = 1$ and $C_Y = 1+\eta'$.
    \label{enu:proof:factorization-b}

  \item $(b_R^{(\varepsilon)} : R\in \dint^m\times\dint^n)$ almost--diagonalizes $T$ so that $T$ has
    large diagonal.  To be more precise, we have the estimates
    \begin{subequations}\label{eq:proof:factorization-c}
      \begin{align}
        \sum_{\substack{R' \in \dint^m\times\dint^n\\R'\neq R}} | \langle b_R^{(\varepsilon)}, T
        b_{R'}^{(\varepsilon)} \rangle | &\leq \eta' \|b_R^{(\varepsilon)}\|_2^2, & R &\in
        \dint^m\times\dint^n,
        \label{eq:proof:factorization-c:a}\\
        \langle b_R^{(\varepsilon)}, T b_R^{(\varepsilon)} \rangle &\geq \delta
        \|b_R^{(\varepsilon)}\|_2^2, & R &\in \dint^m\times\dint^n.
        \label{eq:proof:factorization-c:b}
      \end{align}
    \end{subequations}
    \label{enu:proof:factorization-c}
  \end{enumerate}

  The rest of the proof is exactely as outlined in~\cite[Chapter 5.2]{mueller:2005}.
  Also see~\cite{lechner_mueller:2014} for a specific bi--parameter variant
  following~\cite[Chapter 5.2]{mueller:2005}.
  Additionally, we will keep track of the exact value of our constants.
  Define the subspace $Y$ of $H^p_m(H^q_N)$ (see condition~\eqref{enu:proof:factorization-a}) by
  \begin{equation*}
    Y = \spn\{ b_R^{(\varepsilon)} : R\in \dint^m\times\dint^n\},
  \end{equation*}
  equipped with the $H^p(H^q)$ norm.
  Condition~\eqref{enu:proof:factorization-b} has three implications that we will now record.
  Firstly, for any $1\leq r, s < \infty$ and $1 < r', s' \leq \infty$ with
  $\frac{1}{r} + \frac{1}{r'} = 1$ and $\frac{1}{s} + \frac{1}{s'} = 1$, we have by
  Theorem~\ref{thm:projection} that the operator $E : H^p_m(H^q_n)\to Y$ defined as the
  linear extension of $h_R\mapsto b_R^{(\varepsilon)}$, $R\in \dint^m\times \dint^n$, satisfies
  \begin{equation}\label{eq:proof:factor-1}
    \vcxymatrix{
      H^p_m(H^q_n) \ar[r]^{\Id} \ar[d]_E & H^p_m(H^q_n)\\
      Y \ar[r]_\Id & Y \ar[u]_{E^{-1}}
    }
    \qquad \|E\|\|E^{-1}\|\leq (1+\eta')^{2k},
  \end{equation}
  where $k$ is the integer in Theorem~\ref{thm:projection}.
  Thirdly, by~\eqref{eq:proof:factorization-c:b} together with the projection Theorem~\ref{thm:projection},
  we obtain that the operator $U : H^p(H^q)\to Y$, defined by
  \begin{equation*}
    Uf
    = \sum_{R\in \dint^n\times \dint^n}
    \frac{\langle b_R^{(\varepsilon)}, f \rangle}
    {\langle b_R^{(\varepsilon)}, T b_R^{(\varepsilon)}\rangle}
    b_R^{(\varepsilon)},
    \qquad f\in H^p(H^q),
  \end{equation*}
  satisfies the estimate
  \begin{equation}\label{eq:almost-inverse:bounded}
    \|U f\|_{H^p(H^q)}
    \leq \frac{(1+\eta')^k}{\delta} \|f\|_{H^p(H^q)},
    \qquad f\in H^p(H^q).
  \end{equation}
  For $g = \sum_{R\in \dint^n} a_R b_R^{(\varepsilon)} \in Y$, Theorem~\ref{thm:projection} together
  with~\eqref{eq:proof:factorization-c} yields that
  \begin{equation}\label{eq:almost-inverse:inverse-estimate}
    \|UT g - g\|_{H^p(H^q)}
    \leq \eta'\frac{m n}{(1+\eta') \delta} \|g\|_{H^p(H^q)}.
  \end{equation}
  Let $J : Y\to H^p_m(H^q_N)$ denote the operator given by $Jy = y$.  Define the operator $V :
  H^p_m(H^q_N)\to Y$ by $V=(UTJ)^{-1}U$ (which is well defined by~\eqref{eq:proof:eta'}), and note
  that
  \begin{equation}\label{eq:proof:factor-2}
    \vcxymatrix{
      Y \ar[rr]^\Id \ar[dd]_J & & Y\\
      & Y \ar[ru]^{(UTJ)^{-1}} &\\
      H^p_m(H^q_N) \ar[rr]_T & & H^p_m(H^q_N) \ar[lu]_{U} \ar[uu]_V
    }
    \qquad \|J\|\|V\| \leq \frac{(1+\eta')^k}{\delta - \eta'\frac{mn}{1+\eta'}}.
  \end{equation}
  Merging diagram~\eqref{eq:proof:factor-1} with diagram~\eqref{eq:proof:factor-2} and
  recalling~\eqref{eq:proof:eta'} concludes the proof.
\end{proof}

\begin{rem}\label{rem:thm:localized-factorization}

  Similar to Remark~\ref{rem:thm:quasi-diag} (see also Remark~\ref{rem:thm:projection}), we could replace the
  bi--parameter Haar system $(h_K\otimes h_L)$ in Theorem~\ref{thm:localized-factorization} with a tensor
  product $(e_K\otimes f_L)$ that satisfies~\textref[P]{enu:p1}--\textref[P]{enu:p4} and some
  additional regularity assumptions (see
  Lemma~\ref{lem:product:local-product}~\eqref{enu:lem:product:local-product:i}
  and~\eqref{enu:lem:product:local-product:ii}),
  and simply repeat the proof.
  To be more precise, the large diagonal hypothesis of Theorem~\ref{thm:localized-factorization} would read
  as follows:
  \begin{equation*}
    |\langle e_K\otimes f_L, T e_K\otimes f_L \rangle| \geq \delta \|e_K\otimes f_L\|_2^2,
    \qquad\text{$K\in \dint^m,\ L\in \dint^N$},
  \end{equation*}
  respectively $K\in \dint^M,\ L\in \dint^n$.
\end{rem}

\section{Sums of finite dimensional Banach spaces}\label{sec:diag-glue-sums}

Section~\ref{subsec:diagonalization}, we discuss the necessary tools to diagonalize operators acting on a
direct sum of finite dimensional Banach spaces.
In Section~\ref{subsec:glueing}, we describe how to ``glue together'' factorization results in finite
dimensional Banach spaces, to obtain a factorization result in the direct sum of these spaces.
The proofs of the theorems in Section~\ref{subsec:diagonalization} and Section~\ref{subsec:glueing} have been
repeated in numerous situations see
e.g.~\cite{bourgain:1983, blower:1990, mueller:2005, wark:2007, lechner_mueller:2014}.
This is the author's attempt to avoid repetition in upcoming papers.
In Section~\ref{subsec:isom-non-isom} we discuss isomorphisms and non--isomorphisms of direct sums of
finite dimensional Banach spaces.
Finally, we give proofs of the main results Theorem~\ref{thm:factorization} and Theorem~\ref{thm:primary} in 
Section~\ref{subsec:proof-main-theorem} and Section~\ref{subsec:proof-primary}, respectively.

\subsection{Diagonalization}\label{subsec:diagonalization}\hfill

\noindent
We  briefly discuss two lemmas to diagonalize an operator on a direct sum of finite dimensional
Banach spaces.
The first lemma follows by a gliding hump argument, and is therefore limited to finite parameters in
the direct sum.
The second lemma for direct sums with infinite parameter, uses an additional hypothesis, see
Definition~\ref{dfn:property-pafds}.
For the space $\big(\sum_{n\in \mathbb N} H^p_n(H^q_n)\big)_\infty$, this hypothesis will be realized
by Theorem~\ref{thm:projections-that-annihilate}.

\begin{lem}\label{lem:diagonalization-1}
  Let $1 \leq r < \infty$, and let $(X_n)_{n\in \mathbb N}$ be a non--decreasing sequence of finite
  dimensional Banach spaces.
  Let $X^{(r)} = \big(\sum_{n\in \mathbb N} X_n\big)_r$ and $T : X^{(r)}\to X^{(r)}$ be a bounded
  linear operator.
  For each $\theta > 0$ there exist norm $1$ operators $U, V : X^{(r)}\to X^{(r)}$ such that
  $UV = \Id_{X^{(r)}}$,
  and $\widehat T$ given by $\widehat T = U T V$ is almost diagonal, i.e.
  \begin{equation}\label{eq:lem:diagonalization-1}
    \| \widehat T - \sum_{n=1}^\infty P_n \widehat T P_n \| \leq \theta.
  \end{equation}
  The norm $1$ operator $P_n : X^{(r)}\to X^{(r)}$ denotes the coordinate projection onto $X_n$.
  The above series of operators is understood as a formal series and does not indicate any form of
  convergence.
\end{lem}
We remark that an operator $D : X^{(r)}\to X^{(r)}$ is called \emph{diagonal operator} if
$D = \sum_{n=1}^\infty P_n D P_n$.
The proof of Lemma~\ref{lem:diagonalization-1} is a standard gliding hump argument and therefore
omitted.

Definition~\ref{dfn:property-pafds} is merely a surrogate of the corresponding theorems
in~\cite{bourgain:1983, blower:1990, mueller:2005, wark:2007, lechner_mueller:2014}.
\begin{dfn}\label{dfn:property-pafds}
  We say that a non--decreasing sequence of finite dimensional Banach spaces $(X_n)_{n\in \mathbb N}$
  with $\sup_n \dim X_n = \infty$ has the property that
  \emph{projections almost annihilate finite dimensional subspaces with constant $C_P > 0$}
  if the following conditions are satisfied:

  For all $n,d\in \mathbb N$ and $\eta > 0$ there exists an integer $N=N(n,d,\eta)$ such that for
  any $d$--dimensional subspace $F\subset X_N$ there exists a bounded projection
  $Q : X_N\to X_N$ and an isomorphism $S : X_n\to Q(X_N)$ such that
  \begin{enumerate}[(i)]
  \item $\|Q\| \leq C_P$,
  \item $\|S\|, \|S^{-1}\| \leq C_P$,
  \item $\|Q x\| \leq \eta \|x\|$, for all $x\in F$.
  \end{enumerate}
\end{dfn}

The following diagonalization Lemma~\ref{lem:diagonalization-2} allows us to diagonalize an operator on
direct sums with infinite parameter $r=\infty$, by additionally using the property defined in
Definition~\ref{dfn:property-pafds}.
\begin{lem}\label{lem:diagonalization-2}
  Let $(X_n)_{n\in \mathbb N}$ denote a non--decreasing sequence of finite dimensional Banach spaces
  with $\sup_n \dim X_n = \infty$ having the property that projections almost annihilate finite
  dimensional subspaces with constant $C_P>0$ (see Definition~\ref{dfn:property-pafds}).
  Now put $X^{(\infty)} = \big(\sum_{n\in \mathbb N} X_n\big)_\infty$ and let
  $T : X^{(\infty)}\to X^{(\infty)}$ be a bounded linear operator.
  For each $\theta > 0$ there exist operators $U, V : X^{(\infty)}\to X^{(\infty)}$ such that
  $UV = \Id_{X^{(\infty)}}$, and $\widehat T$ given by $\widehat T = U T V$ is almost diagonal,
  i.e.
  \begin{equation}\label{eq:lem:diagonalization-2}
    \| \widehat T - \sum_{n=1}^\infty P_n \widehat T P_n \| \leq \theta.
  \end{equation}
  The norm $1$ operator $P_n : X^{(\infty)}\to X^{(\infty)}$ denotes the coordinate projection onto
  $X_n$.
  The above series of operators is understood as a formal series and does not indicate any form of
  convergence.
  The operators $U$ and $V$ can be chosen such that $\|U\| \|V\| \leq C_P^3$.
\end{lem}
The proof is completely analogous to the corresponding diagonalization theorems
in~\cite{bourgain:1983, blower:1990, mueller:2005, wark:2007, lechner_mueller:2014}, and we
therefore omit it.

\subsection{Glueing}\label{subsec:glueing}\hfill

Here, we ``glue together'' factorization diagrams for finite dimensional Banach spaces, to obtain a
factorization diagram for the direct sum of these spaces.
We distinguish between finite and infinite parameters.
Again, we refer
to~\cite{bourgain:1983, blower:1990, mueller:2005, wark:2007, lechner_mueller:2014}.

\begin{pro}\label{pro:local}
  Let $(X_n)_{n\in \mathbb N}$ be an increasing sequence of finite dimensional Banach spaces.
  Let $\Gamma > 0$ and $\eta > 0$ be fixed.
  Assume that for each $n\in \mathbb N$ there exists an integer $N=N(n,\Gamma,\eta)$ such that for
  any operator $T_n : X_N\to X_N$ with $\|T_n\|\leq \Gamma$ one can find operators
  $R_n : X_n\to X_N$ and $S_n :X_N\to X_n$ so that
  \begin{equation}\label{eq:pro:local:1}
    \vcxymatrix{
      X_n \ar[r]^{\Id_{X_n}} \ar[d]_{S_n} & X_n\\
      X_N \ar[r]_{H_n} & X_N \ar[u]_{R_n}
    }
  \end{equation}
  where $H_n = T_n$ or $H_n = \Id_{X_N} - T_n$, and $\|R_n\|\|S_n\|\leq 1+\eta$.

  Let $1 \leq r \leq \infty$, put $X^{(r)} = \big(\sum_{n\in \mathbb N} X_n\big)_r$ and let
  $T : X^{(r)}\to X^{(r)}$ be a bounded, linear operator with $\|T\|\leq \Gamma$.
  If $r=\infty$ (and only then) we assume additionally that $(X_n)_{n\in \mathbb N}$ has the
  property that projections almost annihilate finite dimensional subspaces with constant $C_P>0$
  (see Definition~\ref{dfn:property-pafds}).
  
  Then there exist operators $P, Q : X^{(r)}\to X^{(r)}$ such that
  \begin{equation}\label{eq:pro:local:2}
    \vcxymatrix{
      X^{(r)} \ar[r]^{\Id_{X^{(r)}}} \ar[d]_P & X^{(r)}\\
      X^{(r)} \ar[r]_H & X^{(r)} \ar[u]_Q
    }
  \end{equation}
  for $H = T$ or $H = \Id_{X^{(r)}} - T$.
  For each $\varepsilon > 0$ the operators $P$ and $Q$ can be chosen so that
  $\|P\| \|Q\| \leq 1 + \eta + \varepsilon$, if $r < \infty$, and for $r=\infty$ we obtain
  $\|P\| \|Q\| \leq C_P^3(1 + \eta + \varepsilon)$.
\end{pro}

\begin{proof}
  For a proof see e.g.~\cite{blower:1990, lechner_mueller:2014}.
\end{proof}

\subsection{Isomorphisms and non--isomorphisms}\label{subsec:isom-non-isom}\hfill

Here, we briefly discuss two results on sums of finite dimensional Banach spaces.
Together, they show us that $\big(\sum_{m,n} H^p_m(H^q_n)\big)_r$ is isomorphic to
$\big(\sum_n H^p_n(H^q_n)\big)_s$ if and only if $r=s$.
The same is true for $\big(\sum_{m,n} H^p_m(H^q_n)^*\big)_r$ and
$\big(\sum_n H^p_n(H^q_n)^*\big)_s$.

The following Proposition~\ref{pro:iso-iso} is a simple consequence of Pe{\l}czy{\'n}ski's decomposition
method.
Therefore, we omit the proof.
\begin{pro}\label{pro:iso-iso}
  Let $1\leq r \leq \infty$ and let $(X_{m,n} : m,n\in \mathbb N)$ denote a sequence of finite
  dimensional Banach spaces such that $X_{m,n}\subset X_{m+1,n}$ and $X_{m,n}\subset X_{m,n+1}$.
  Then the space $\big(\sum_{m,n\in \mathbb N} X_{m,n}\big)_r$ is isometrically isomorphic to
  $\big(\sum_{n\in \mathbb N} X_{n,n}\big)_r$.
\end{pro}

Theorem~\ref{thm:pitt} is a finite dimensional Banach space variant of Pitt's theorem.
See also~\cite{defant:lopez-molina:rivera:2000}.
\begin{thm}\label{thm:pitt}
  Let $1 \leq r,s \leq \infty$, and let $(X_n)_{n\in \mathbb N}$ denote an increasing sequence of
  finite dimensional Banach spaces.
  Let $X^{(r)}$ denote $\big(\sum_{n\in \mathbb N} X_n\big)_r$ and let $X^{(s)}$ denote the space
  $\big(\sum_{n\in \mathbb N} X_n\big)_s$.
  If $T : X^{(r)}\to X^{(s)}$ is an isomorphism, then $r=s$.
  Consequently, all the spaces $X^{(r)}$, $1\leq r \leq \infty$ are mutually non-isomorphic.
\end{thm}

The proof is a standard gliding hump argument for $1\leq r,s < \infty$.
The remaining cases follow immediately by considering the separability/non--separability of the respective spaces.
For those reasons, we omit the proof.

\subsection{Proof of the main result Theorem~\ref{thm:factorization}}\label{subsec:proof-main-theorem}\hfill

\noindent
For convenience, we reassert Theorem~\ref{thm:factorization} here.
\begin{thm}[Main result Theorem~\ref{thm:factorization}]\label{thm:factorization:restated}
  Let $1\leq p,q < \infty$ and $1\leq r \leq \infty$,
  and for all $n\in \mathbb N$ let $X_n$ denote the space $H^p_n(H^q_n)$ or its dual
  $H^p_n(H^q_n)^*$.
  For any $\eta > 0$ and any operator
  $T : \big(\sum_{n\in \mathbb N} X_n\big)_r\to
  \big(\sum_{n\in \mathbb N} X_n\big)_r$,
  there exist operators
  $R,S : \big(\sum_{n\in \mathbb N} X_n\big)_r\to \big(\sum_{n\in \mathbb N} X_n\big)_r$
  such that
  \begin{equation}\label{eq:thm:factorization:restated}
    \vcxymatrix{
      \big(\sum_{n\in \mathbb N} X_n\big)_r \ar[r]^{\Id} \ar[d]_S
      & \big(\sum_{n\in \mathbb N} X_n\big)_r\\
      \big(\sum_{n\in \mathbb N} X_n\big)_r \ar[r]_H
      & \big(\sum_{n\in \mathbb N} X_n\big)_r \ar[u]_R
    }
  \end{equation}
  for $H = T$ or $H = \Id - T$ and $\|R\| \|S\| \leq 2 + \eta$.
\end{thm}

The following Ramsey type Theorem~\ref{thm:ramsey} is the last missing ingredient for the proof of Theorem~\ref{thm:factorization:restated} (Main result Theorem~\ref{thm:factorization}).
\begin{thm}\label{thm:ramsey}
  Given $n_0 \in \mathbb N$ there exists $n \in \mathbb N$ such that for any collection
  $\mathscr C \subset \dint^n\times \dint^n$ one finds $\mathscr A,\mathscr B \subset \dint$
  satisfying
  \begin{enumerate}[(i)]
  \item $\mathscr A \times \mathscr B \subset \mathscr C$
    or
    $\mathscr A \times \mathscr B \subset (\dint^n\times \dint^n) \setminus \mathscr C$,
  \item $\cc{\mathscr A} \geq n_0$ and $\cc{B} \geq n_0$.
  \end{enumerate}
  One can choose $n = n_0\, 2^{4^{n_0}}$.
\end{thm}
\begin{proof}
  We refer to~\cite{mueller:1994}.
  See also~\cite{lechner_mueller:2014}.
\end{proof}

\begin{proof}[Proof of Theorem~\ref{thm:factorization:restated} (Main result Theorem~\ref{thm:factorization})]
  The proof follows the pattern of the corresponding proof in~\cite{lechner_mueller:2014}.
  Let $1 \leq p,q < \infty$, $1\leq r \leq \infty$ and define the space $X^{(r)}$ by
  \begin{equation*}
    X^{(r)} = \big( \sum_{n\in \mathbb N} X_n \big)_r.
  \end{equation*}
  Let $\eta > 0$ and $T : X^{(r)}\to X^{(r)}$.
  Again, we will only prove the case where $X_n = H^p_n(H^q_n)$.
  The case $X_n = H^p_n(H^q_n)^*$ is repeating the following argument with the roles of $T$ and
  $T^*$ reversed.

  In the first part of the proof we will show that for all $n\in \mathbb N$ and $\Gamma > 0$, there
  is an integer $N = N(n,\Gamma,\eta)$ such that for any operator $D_n : X_N\to X_N$ with
  $\|D_n\|\leq \Gamma$ there exist operators $R_n$, $S_n$ so that
  \begin{equation}\label{eq:proof:thm:factorization-1}
    \vcxymatrix{
      X_n \ar[r]^{\Id} \ar[d]_{S_n} & X_n\\
      X_N \ar[r]_{H_n} & X_N \ar[u]_{R_n}
    }
    \qquad \|R_n\|\|S_n\|\leq 2 + \eta,
  \end{equation}
  where $H_n = D_n$ or $H_n = \Id_{X_N} - D_n$.

  To this end, let $\eta' > 0$ be parameter, which will be specified at a later point.
  Firstly, we choose $\gamma = \gamma(n,\eta')$ so large, so that for any collection
  $\mathscr E\subset \dint$ with Carleson constant $\cc{\mathscr E} \geq \gamma$ exist collections
  $\mathscr E_I \subset \mathscr  E$, $I\in \dint^n$, and an affine map $\psi : [0,1)\to [0,1)$ so
  that the sequence of collections
  $(\psi(\mathscr E_I)\times \psi(\mathscr E_J) : I,J \in \dint^n)$ satisfies
  \textref[P]{enu:p1}--\textref[P]{enu:p4} with constant $C_X = 1+\eta'$, as well as the additional
  regularity assumptions~\eqref{enu:lem:product:local-product:i}
  and~\eqref{enu:lem:product:local-product:ii} of Lemma~\ref{lem:product:local-product}.
  For a detailed exposition we refer the reader to~\cite{mueller:2005}.

  Secondly, if we put $n_1 = \lceil \gamma 2^{4^\gamma} \rceil$, the Ramsey Theorem~\ref{thm:ramsey}
  asserts that whenever $\mathscr C\subset \dint^{n_1}\times \dint^{n_1}$, there exist
  collections $\mathscr E, \mathscr F\subset \dint^{n_1}$ with
  $\cc{\mathscr E},\cc{\mathscr F}\geq \gamma$ so that either
  \begin{equation*}
    \mathscr E\times \mathscr F\subset \mathscr C
    \qquad\text{or}\qquad
    \mathscr E\times \mathscr F\subset \dint^{n_1}\times\dint^{n_1}\setminus \mathscr C.
  \end{equation*}

  Thirdly, applying Theorem~\ref{thm:quasi-diag} with $\delta = 0$, yields an integer
  $N=N(n_1,\Gamma,\eta')$ (this is exactely the integer $N$ of Theorem~\ref{thm:quasi-diag} with the
  specified parameters) and a sequence of collections of sets $(\mathscr B_R : R\in \dint)$ with the
  following properties:
  \begin{enumerate}[(a)]
  \item $\mathscr B_R \subset \dint^N\times\dint^N$ for all $R\in \dint^{n_1}\times\dint^{n_1}$.
    \label{proof:enu:thm:factorization:quasi-diag-i}

  \item $(\mathscr B_R : R\in \dint^n\times\dint^n)$ satisfies the local product conditions with
    constants $C_X = 1+\eta'$ and $C_Y = 1+\eta'$.
    \label{proof:enu:thm:factorization:quasi-diag-ii}

  \item The $(b_R : R\in \dint^n\times\dint^n)$ almost--diagonalize $D_n$.
    To be more precise, we have the estimate
    \begin{equation}\label{proof:eq:thm:factorization:quasi-diag-iii}
      \sum_{\substack{R' \in \dint^n\times\dint^n\\R'\neq R}}
      | \langle b_R, D_n b_{R'} \rangle |
      \leq \eta' \|b_R\|_2^2,
      \qquad R \in \dint^n\times\dint^n.
    \end{equation}
  \end{enumerate}
  We note that since there is no lower estimate for the diagonal, we can choose all the signs
  $\varepsilon_Q$ equal to $1$, so henceforth we will omit the superscript $(\varepsilon)$ of
  $b_R^{(\varepsilon)}$, and simply denote the function by $b_R$.
 
  Fourthly, we will now combine the first three steps.
  We specify the collection of dyadic rectangles $\mathscr C$ by
  \begin{equation}\label{eq:ramsey-collection}
    \mathscr C
    = \{ R\in \dint^{n_1}\times\dint^{n_1} :
    |\langle b_R, D_n b_R\rangle|\geq \|b_R\|_2^2/2
    \}.
  \end{equation}
  By the choice of our parameters in the first two steps, we can find finite sequences of
  collections $(\mathscr E_I : I\in \dint^n)$ and $(\mathscr F_J : J\in \dint^n)$ so that
  \begin{equation*}
    \mathscr E_I\times \mathscr F_J \subset \mathscr C
    \qquad\text{or}\qquad
    \mathscr E_I\times \mathscr F_J \subset \dint^{n_1}\times \dint^{n_1}\setminus \mathscr C,
  \end{equation*}
  for all $I,J\in \dint^{n_1}$.  If the first inclusion is true we put $H_n = D_n$, if the second is
  true, then we define $H_n = \Id_{X_N} - D_n$.
  We will now construct a block basis $(\widetilde b_R)$ of the block
  basis $(b_R)$ of the Haar system $(h_R)$.
  We define the collection of dyadic rectangles $\widetilde{\mathscr B}_{I\times J}$ by
  \begin{equation}\label{eq:iterated-collection}
    \widetilde{\mathscr B}_{I\times J}
    = \bigcup_{\substack{E\in \mathscr E_I\\F\in \mathscr F_J}} \mathscr B_{E\times F},
    \qquad I,J\in \dint^n,
  \end{equation}
  and the corresponding block basis elements $\widetilde b_{I\times J}$ by
  \begin{equation}\label{eq:iterated-block-basis}
    \widetilde b_R
    =  \sum_{Q\in \widetilde{\mathscr B}_R} h_Q,
    \qquad R\in \dint^n\times \dint^n.
  \end{equation}
  Note that $\widetilde{\mathscr B}_R\subset \dint^N\times \dint^N$, $R\in \dint^n\times \dint^n$
  hence $\widetilde b_R\in H^p_N(H^q_N)$, $R\in \dint^n\times \dint^n$.
  The reiteration Lemma~\ref{lem:product:local-product} gives us that
  $\widetilde{\mathscr B}_R$, $R\in \dint^n\times \dint^n$ satisfies the local product
  conditions~\textref[P]{enu:p1}--\textref[P]{enu:p4} with constants $C_X = C_Y = (1+\eta')^4$.
  Now put
  \begin{equation*}
    Y_n = \spn\{\widetilde b_R : R\in \dint^n\times\dint^n\}\subset X_N,
  \end{equation*}
  equipped with the $H^p(H^q)$ norm.
  We summarize what we have proved this far: by Theorem~\ref{thm:projection}, we have that
  \begin{equation}\label{eq:factor-1}
    \vcxymatrix{
      X_n \ar[r]^{\Id_{X_n}} \ar[d]_{E_n} & X_n\\
      Y_n \ar[r]_{\Id_{Y_n}} & Y_n \ar[u]_{E_n^{-1}}
    }
    \qquad \|E_n\|\|E_n^{-1}\|\leq (1+\eta')^{8k},
  \end{equation}
  where $k$ is the integer appearing in Theorem~\ref{thm:projection}.
  Furthermore, by~\eqref{proof:eq:thm:factorization:quasi-diag-iii}, \eqref{eq:ramsey-collection},
  \eqref{eq:iterated-collection}, \eqref{eq:iterated-block-basis} and Theorem~\ref{thm:projection} we have
  the estimates
  \begin{subequations}\label{eq:factor-estimate}
    \begin{align}
      \sum_{\substack{R' \in \dint^n\times\dint^n\\R'\neq R}} | \langle \widetilde b_R, H_n
      \widetilde b_{R'} \rangle | &\leq c(n,\eta') \|\widetilde b_R\|_2^2, & R &\in
      \dint^n\times\dint^n,
      \label{eq:factor-estimate:a}\\
      |\langle \widetilde b_R, H_n \widetilde b_R \rangle| &\geq \big( \frac{1}{2} - c(n,\eta')
      \big) \|\widetilde b_R\|_2^2, & R &\in \dint^n\times\dint^n,
      \label{eq:factor-estimate:b}
    \end{align}
  \end{subequations}
  where $c(n,\eta')\to 0$, if $\eta'\to 0$.
  Remark Remark~\ref{rem:thm:localized-factorization} allows us to replace $b_R^{(\varepsilon)}$ in
  Theorem~\ref{thm:localized-factorization} by $\widetilde b_R$, thus Theorem~\ref{thm:localized-factorization}
  yields
  \begin{equation*}
    \vcxymatrix{Y_n \ar[r]^{\Id_{Y_n}} \ar[d]_E & Y_n\\
      H^p_m(H^q_N) \ar[r]_{H_n} & H^p_m(H^q_N) \ar[u]_P}
  \end{equation*}

  shows
  that~\eqref{eq:proof:thm:factorization-1} is true, if we choose $\eta'$ appropriately.
  Alternatively to invoking Remark~\ref{rem:thm:localized-factorization} and
  Theorem~\ref{thm:localized-factorization}, we could repeat the proof of
  Theorem~\ref{thm:localized-factorization} after~\eqref{eq:proof:factorization-c} with $\widetilde b_R$
  instead of $b_R^{(\varepsilon)}$
  (which is of course part of the argument behind Remark~\ref{rem:thm:localized-factorization}).

  Finally, observe that Theorem~\ref{thm:projections-that-annihilate} implies that $(X_n)_{n\in \mathbb N}$
  has the property that projections almost annihilate finite dimensional subspaces with constant
  $1 + \eta'$ (see Definition~\ref{dfn:property-pafds}).
  Thus, applying Proposition~\ref{pro:local} concludes the proof.
\end{proof}

\subsection{Proof of the main result Theorem~\ref{thm:primary}}\label{subsec:proof-primary}\hfill

\noindent
We will now give the proof Theorem~\ref{thm:primary}.
It follows from Theorem~\ref{thm:factorization} and Pe{\l}czy{\'n}ski's decomposition method.
\begin{proof}
  Let $1\leq r \leq \infty$ and let $X$ denote either the space
  $\big(\sum_{n\in \mathbb N} H^p_n(H^q_n)\big)_r$ or
  $\big(\sum_{n\in \mathbb N} H^p_n(H^q_n)^*\big)_r$.
  Let $Q$ be a bounded projection on $X$.

  Clearly, $(\sum X)_r$ is isomorphic to $(\sum (\sum X)_r)_r$.
  Furthermore, $X$ is isomorphic to a complemented subspace of $(\sum X)_r$, and $(\sum X)_r$ is
  isomorphic to a complemented subspace of $X$.  Hence, by Pe{\l}czy{\'n}ski's decomposition method,
  $X$ is isomorphic to $(\sum X)_r$.
  From Theorem~\ref{thm:factorization}, we obtain that
  \begin{equation*}
    \vcxymatrix{
      X \ar[r]^{\Id} \ar[d]_S & X\\
      X \ar[r]_H & X \ar[u]_R
    }
  \end{equation*}
  where $H = Q$ or $H = \Id - Q$.
  The diagram shows that $H(X)$ is a complemented subspace of $X$, and that $X$ is isomorphic to a
  complemented subspace of $H(X)$, hence, by Pe{\l}czy{\'n}ski's decomposition method we obtain that
  $H(X)$ is isomorphic to $X$.
\end{proof}

\bibliographystyle{abbrv}
\bibliography{bibliography}

\begin{thebibliography}{10}

\bibitem{andrew:1979}
A.~D. Andrew.
\newblock Perturbations of {S}chauder bases in the spaces {$C(K)$} and
  {$L^{p}$}, {$p>1$}.
\newblock {\em Studia Math.}, 65(3):287--298, 1979.

\bibitem{blower:1990}
G.~Blower.
\newblock The {B}anach space {$B(l^2)$} is primary.
\newblock {\em Bull. London Math. Soc.}, 22(2):176--182, 1990.

\bibitem{bourgain:1983}
J.~Bourgain.
\newblock On the primarity of {$H^{\infty }$}-spaces.
\newblock {\em Israel J. Math.}, 45(4):329--336, 1983.

\bibitem{capon:1982}
M.~Capon.
\newblock Primarit\'e de {$L^{p}(L^{r})$}, {$1<p,\,r<\infty$}.
\newblock {\em Israel J. Math.}, 42(1-2):87--98, 1982.

\bibitem{defant:lopez-molina:rivera:2000}
A.~Defant, J.~A. L{\'o}pez-Molina, and M.~J. Rivera.
\newblock On {P}itt's theorem for operators between scalar and vector-valued
  quasi-{B}anach sequence spaces.
\newblock {\em Monatsh. Math.}, 130(1):7--18, 2000.

\bibitem{laustsen:lechner:mueller:2015}
N.~J. {Laustsen}, R.~{Lechner}, and P.~F.~X. {M{\"u}ller}.
\newblock {Factorization of the identity through operators with large
  diagonal}.
\newblock unpublished.

\bibitem{lechner_mueller:2014}
R.~Lechner and P.~F.~X. M{\"u}ller.
\newblock Localization and projections on bi-parameter {BMO}.
\newblock {\em Q. J. Math.}, 66(4):1069--1101, 2015.

\bibitem{lindenstrauss-tzafriri:1977}
J.~Lindenstrauss and L.~Tzafriri.
\newblock {\em Classical {B}anach spaces. {I}}.
\newblock Springer-Verlag, Berlin-New York, 1977.
\newblock Sequence spaces, Ergebnisse der Mathematik und ihrer Grenzgebiete,
  Vol. 92.

\bibitem{maurey:1980}
B.~Maurey.
\newblock Isomorphismes entre espaces {$H_{1}$}.
\newblock {\em Acta Math.}, 145(1-2):79--120, 1980.

\bibitem{mueller:1994}
P.~F.~X. M{\"u}ller.
\newblock Orthogonal projections on martingale {$H^1$} spaces of two
  parameters.
\newblock {\em Illinois J. Math.}, 38(4):554--573, 1994.

\bibitem{mueller:2005}
P.~F.~X. M{\"u}ller.
\newblock {\em Isomorphisms between {$H\sp 1$} spaces}, volume~66 of {\em
  Instytut Matematyczny Polskiej Akademii Nauk. Monografie Matematyczne (New
  Series) [Mathematics Institute of the Polish Academy of Sciences.
  Mathematical Monographs (New Series)]}.
\newblock Birkh\"auser Verlag, Basel, 2005.

\bibitem{wark:2007}
H.~M. Wark.
\newblock The {$l^\infty$} direct sum of {$L^p$} {$(1<p<\infty)$} is primary.
\newblock {\em J. Lond. Math. Soc. (2)}, 75(1):176--186, 2007.

\bibitem{wojtaszczyk:1991}
P.~Wojtaszczyk.
\newblock {\em Banach spaces for analysts}, volume~25 of {\em Cambridge Studies
  in Advanced Mathematics}.
\newblock Cambridge University Press, Cambridge, 1991.

\end{thebibliography}

\end{document}